\documentclass{article}
\usepackage{amssymb,amsmath,amsfonts}

\newtheorem{theorem}{Theorem}

\newtheorem{corollary}[theorem]{Corollary}

\newtheorem{definition}[theorem]{Definition}
\newtheorem{example}[theorem]{Example}

\newtheorem{lemma}[theorem]{Lemma}

\newtheorem{proposition}[theorem]{Proposition}
\newtheorem{remark}[theorem]{Remark}

\newenvironment{proof}[1][Proof]{\noindent\textbf{#1.} }{\ \rule{0.5em}{0.5em}}

\begin{document}

\title{Connections between commutative rings and some algebras of logic}
\author{Cristina Flaut and Dana Piciu}
\date{}
\maketitle

\begin{abstract}
In this paper using the connections between some subvarieties of residuated
lattices, we investigated some properties of the lattice of ideals in
commutative and unitary rings. We give new characterizations for commutative
rings $A$ in which $Id(A)$ is an MV-algebra, a Heyting algebra or a Boolean
algebra and we establish connections between these types of rings. We are
very interested in the finite case and we present summarizing statistics. We
show that the lattice of ideals in a finite commutative ring of the form $A=%
\mathbb{Z}_{k_{1}}\times \mathbb{Z}_{k_{2}}\times ...\times \mathbb{Z}%
_{k_{r}},$ where $k_{i}=p_{i}^{\alpha _{i}}$ and $p_{i}$ a prime number, for
all $i\in \{1,2,...,r\},$ \ is a Boolean algebra or an MV-algebra (which is
not Boolean).

Using this result we generate the binary block codes associated to the
lattice of ideals in finite commutative rings and we present a new way to
generate all (up to an isomorphism) finite MV-algebras using rings.


\end{abstract}

\section{\textbf{Introduction}}

Residuated lattices were introduced by Dilworth and Ward, through the papers
[Di; 38], [WD; 39].\medskip\ The study of residuated lattices is originated
in 1930 in the context of theory of rings, with the study of ring ideals. It
is known that the lattice of ideals of a commutative and unitary ring is a
residuated lattice. Based on this result, many researchers ([BN; 09], [Bl;
53], [CL; 19], [Pi; 07] and [TT; 22]) have been interested in this
construction.

In this paper, using the connections between some subvarieties of residuated
lattices, we investigated properties of the lattice of ideals in commutative
and unitary rings. We are very interested in the finite case. An argument
for the importance of this case comes from computational consideration.

In general, a solution that is computational tractable consists in
considering algebras with a reasonable small number of elements.

This makes us study in this paper, finite residuated lattices and finite
rings.

In whole this
paper, by a finite commutative unitary ring, we will understand a finite
ring of the form $A=\mathbb{Z}_{k_{1}}\times \mathbb{Z}_{k_{2}}\times
...\times \mathbb{Z}_{k_{r}},$ where $k_{i}=p_{i}^{\alpha _{i}}$ and $p_{i}$
a prime number, for all $i\in \{1,2,...,r\}.$ The paper is organized as follows:

Section 2 contains basic properties and definitions that we use in the paper.

In Section 3 we give new characterization for commutative rings for which
the lattice of ideals is an MV-algebra (Theorem \ref{Theorem 3.4},
Corollaries \ref{Corollary 3.6}, \ref{Corollary 3.8} and \ref{Corollary 3.13}%
), a Heyting algebra (Theorem \ref{Theorem 3.18}) or a Boolean algebra
(Corollaries \ref{Corollary 3.23}, \ref{Corollary 3.25}, \ref{Corollary 3.32}
and \ref{Corollary 3.34}). Also, we establish new connections between these
rings (Corollaries \ref{Corollary 3.23} and \ref{Corollary 3.34}).

We introduce the notion of Chang ring, as commutative unitary ring which
satisfies Chang property: $I+J=(J:(J:I)),$ for every $I,J\in Id(A)$ and we
show that any finite commutative ring has Chang property, that is, its
lattice of ideals is a Boolean algebra or an MV-algebra (which is not
Boolean), see Corollary \ref{Corollary 3.30}.

Also, we prove that if $A$ is a commutative ring, then $Id(A)$\textit{\ }is
a Boolean algebra if and only if\textit{\ }$A$ is a Von Neumann regular ring
satisfying Chang property, or equivalent, if and only if, \textit{\ }$A$ is
a Von Neumann regular ring in which $Ann(Ann(I))=I,$ for every $I\in Id(A),$
see Corollaries \ref{Corollary 3.25} and \ref{Corollary 3.34}$\mathit{.}$
Moreover, if $A$\ is finite, then $Id(A)$\ is a Boolean algebra if and only
if $A$ is a Von Neumann regular ring, see Corollary \ref{Corollary 3.25}.

One of recent applications of residuated lattices is given by Coding Theory.
In Section 4, using the results obtained in Section 3, we describe the
binary block codes associated to the lattice of ideals in a finite
commutative ring with $n$ ideals. In particular, we construct these binary
block codes for $n=4,6,8.$

In Section 5, we present a new way to generate all (up to an isomorphism)
finite MV-algebras using commutative rings and we present summarizing
statistics.

\section{\textbf{Preliminaries}}

Let $A$ be a commutative unitary ring.

The set $Id\left( A\right) $ denotes the set of all ideals of the ring $A$.

We denote by $<x>$ \ the ideal of $A$ generated by $x\in A$.

Let $I,J\in Id\left( A\right) $. The following sets are also ideals in
the ring $A$:
\begin{equation*}
I+J=<I\cup J>=\{i+j,i\in I,j\in J\}\text{, the sum of two ideals = the ideal
generated by }I\cup J\text{;}
\end{equation*}%
\begin{equation*}
I\otimes J=\{\underset{i=1}{\overset{n}{\sum }}f_{i}g_{i},f_{i}\in
I,g_{i}\in J\}\text{, the product of two ideals;}
\end{equation*}%
\begin{equation*}
\left( I:J\right) =\{x\in A,x\cdot J\subseteq I\}\text{, the quotient of two
ideals;}
\end{equation*}%
\begin{equation*}
Ann\left( I\right) =\left( \mathbf{0}:I\right) \text{, the annihilator of
the ideal }I\text{,}
\end{equation*}%
where\textbf{\ }$\mathbf{0}=<0>.$

\begin{remark}
\label{Remark 2.1} ([BP; 02] ) Let $A$ be a commutative unitary ring and $%
I,J,K\in Id\left( A\right) $. Then the following hold:

1) $I\otimes J\subseteq I+J;$

2) $Ann\left( I\right) =\{x\in A,x\cdot I=0\}=\{x\in A,x\cdot i=0,$ for all $%
i\in I\}$;

3) $\underset{x\in A,x\neq 0}{\bigcup Ann(}<x>)=$ the set of all zero
divisors of the ring $A;$

4) $\ Ann\left( \mathbf{0}\right) =A$ and \ $Ann\left( A\right) =\mathbf{0}$;

5) $\ Ann\left( I+J\right) =Ann\left( I\right) \cap Ann\left( J\right) $;

6) $\left( A:I\right) =A$, $\left( I:A\right) =I~$and $\left( I:I\right) =A$.

7) $\ I\subseteq \left( I:J\right) ;$

8) $\left( I:J\right) \otimes J\subseteq I;$

9) $\ I\subseteq J\Leftrightarrow \left( J:I\right) =A$;

10) $\ ((I:J):K)=\ (I:(J\otimes K))=\ ((I:K):J);$

11) $\ (K:(I+J))=\ (K:I)\cap (K:J).$
\end{remark}

\begin{definition}
\label{Definition 2.2} ([BP; 02] ) Let $A$ be a commutative unitary ring and
$I,J\in Id\left( A\right) $. \ The ideals $I$ and $J$ are called \textit{%
coprime} if $I+J=A$, that means there are $i\in I$, $j\in J$ such that $%
i+j=1 $.
\end{definition}

\begin{remark}
\label{Remark 2.3}

1) For a commutative ring $A$, if $I,J\in Id\left( A\right) $ are coprime
ideals, then $I\otimes J=I\cap J$, see [BP; 02], Prop. 2.19.

2) If an ideal $I$ is coprime with $Ann\left( I\right) $, therefore $I\cap
Ann\left( I\right) =\{0\}$. Indeed, since $I\cap Ann\left( I\right)
=I\otimes Ann\left( I\right) $, if $x\in I\cap Ann\left( I\right) $, we have
$x=\underset{i=1}{\overset{n}{\sum }}f_{i}g_{i},~f_{i}\in I,g_{i}\in
Ann\left( I\right) $, therefore $x=0$.
\end{remark}

\begin{definition}
\label{Definition 2.4} ([BD;74]) The partially ordered set $\left( \mathcal{L%
},\leq \right) $ is a \textit{lattice} if for each two elements$~x,y\in
\mathcal{L}$ their supremum and infimum elements exist, denoted by%
\begin{equation*}
sup\{x,y\}=x\vee y\text{ and }inf\{x,y\}=x\wedge y.
\end{equation*}

The lattice $\left( \mathcal{L},\leq \right) $ is a \textit{distributive
lattice } if for each elements $x,y,z\in \mathcal{L}$ we have the following
relation:%
\begin{equation*}
x\wedge \left( y\vee z\right) =(x\wedge y)\vee (x\wedge z).
\end{equation*}

A lattice $\left( \mathcal{L},\leq \right) $ is a \textit{bounded lattice}
if there are the elements $0$ and $1$, the least element in $\mathcal{L}$,
respectively, the greatest element in $\mathcal{L}.$

With the above notations, for a lattice $\left( \mathcal{L},\leq \right) $
an element $x\in \mathcal{L}$ has a \textit{complement }if there is an
element $y\in \mathcal{L}$ satisfying the following relations:
\begin{equation*}
x\vee y=1~\text{and}~x\wedge y=0.
\end{equation*}%
In this situation, the element $x$ is called \textit{complemented}. A
complement of an element is not unique, but, if $\left( \mathcal{L},\leq
\right) $ is distributive, then each element has at most a complement.

The lattice $\left( \mathcal{L},\leq \right) $ is a \textit{complemented
lattice} if it is a bounded lattice and each element $x\in \mathcal{L}$ has
a complement.
\end{definition}

\begin{definition}
\label{Definition 2.5}\ ([BD;74]) An algebra $\left( \mathcal{B},\vee
,\wedge ,^{\prime },0,1\right) $ $~$is called a \textit{Boolean algebra} if $%
\left( \mathcal{B},\vee ,\wedge ,0,1\right) ~$is a bounded distributive
lattice and a complemented lattice in which
\begin{equation*}
b\vee b^{\prime }=1\text{ and }b\wedge b^{\prime }=0,
\end{equation*}%
for all elements $b\in \mathcal{B}$.
\end{definition}

\begin{definition}
\label{Definition 2.6}\ ([I; 09]) A \textit{BCK-algebra} is a structure $%
(X,\leq ,\rightarrow ,1)$ where $(X,\leq )$ is a poset with a greatest
element $1$ and $\rightarrow $ is a binary operation on $L$ such that:

BCK1) $x\rightarrow y\leq (y\rightarrow z)\rightarrow (x\rightarrow z);$

BCK2) $x\leq (x\rightarrow y)\rightarrow y;$

BCK3) $\ \ x\leq y$ iff $x\rightarrow y=1,$ for every $x,y,z\in X.$
\end{definition}

Obviously, by a bounded BCK-algebra we mean a BCK-algebra with a least
element $0.$

For other details regarding BCK-algebras, the readers are referred to [AAT;
96], [Me-Ju; 94].\medskip

Fundamental examples of BCK-algebras come from algebras of logic:

\begin{definition}
\label{Definition 2.7} ([Di; 38], [WD; 39])\textbf{\ \ }A \emph{residuated
lattice} \ is an algebra $(L,\wedge ,\vee ,\odot ,\rightarrow ,0,1)$
equipped with an order $\leq $\ satisfying the following axioms:

LR1) $\ (L,\wedge ,\vee ,0,1)$ is a bounded lattice;

LR2) $\ \ (L,\odot ,1)$ is a commutative ordered monoid;

LR3) $\ \odot $ and $\rightarrow $\ form an \textit{adjoint pair}, i.e., $%
z\leq x\rightarrow y$\ iff $x\odot z\leq y,$\ for all $x,y,z\in L$.
\end{definition}

It is obviously that if $(L,\vee ,\wedge ,\odot ,\rightarrow ,0,1)$ is a
residuated lattice, then $(L,\rightarrow ,0,1)$ is a bounded BCK-algebra.

\begin{remark}
\label{Remark 2.8} ([T; 99]) Let $(\mathcal{B},\wedge ,\vee ,^{\prime },0,1)$
be a \emph{Boolean algebra}. \ If we define for every $x,y\in \mathcal{B}%
,x\odot y=x\wedge y$ and $x\rightarrow y=x^{\prime }\vee y,$ then $(\mathcal{%
B},\wedge ,\vee ,\odot ,\rightarrow ,0,1)$ becomes a residuated lattice.
\end{remark}

\begin{proposition}
\label{Proposition 3.33} ([BD; 74], [COM; 00])\textit{\ Let }$(L,\vee
,\wedge ,\odot ,\rightarrow ,0,1)$\textit{\ be a residuated lattice. The
following assertions are equivalent:}

\textit{(i) }$\ x^{2}=x$ and \textit{\ }$x^{\ast \ast }=x,$\textit{\ for
every \thinspace }$x\in L;$

(ii) $\ x^{2}=x$ \textit{for every \thinspace }$x\in L$ and \textit{\ }$[$%
for $y\in L,$ $y^{\ast }=0\Rightarrow y=1]$;

\textit{(iii) }$\ \ x\vee x^{\ast }=1,$ \textit{\ for every \thinspace }$%
x\in L.$
\end{proposition}

In a residuated lattice $(L,\wedge ,\vee ,\odot ,\rightarrow ,0,1)$ we
consider the identities:

\begin{equation*}
(div)\qquad x\odot (x\rightarrow y)=x\wedge y\qquad \text{ (\textit{%
divisibility)}};
\end{equation*}%
\begin{equation*}
(DN)\qquad x^{\ast \ast }=x\qquad \text{ (\textit{double negation condition)}%
}.
\end{equation*}

\begin{definition}
\label{Definition 2.8.1} ([I; 09], [T; 99]) \ The residuated lattice $L$ is
called:

(i) \emph{divisible}\textit{\ }if $L$ verifies $(div);$

(ii) \emph{involutive} or \emph{Girard monoid} if $L$ verifies $(DN).$
\end{definition}

\begin{definition}
\label{Definition 2.9} ([CHA; 58], [Mu; 07]) An abelian monoid $\left(
M,\oplus ,0\right) $ is called an \textit{MV-algebra} if we have an unary
operation $^{\ast }$ on $M$ such that:

MV1) $(x^{\ast })^{\ast }=x;$

MV2) $x\oplus 0^{\ast }=0^{\ast };$

MV3) $\left( x^{\ast }\oplus y\right) ^{\ast }\oplus y=$ $\left( y^{\ast
}\oplus x\right) ^{\ast }\oplus x$, for all $x,y\in M.$
\end{definition}

Usually, we denote an MV-algebra by $\left( M,\oplus ,^{\ast },0\right) .$

In an MV-algebra $M,$ the constant element $0^{\ast }$ is denoted with $1$,
that means

\begin{equation*}
1=0^{\ast },
\end{equation*}%
and the following auxiliary operations are also defined:

\begin{equation*}
x\odot y=\left( x^{\ast }\oplus y^{\ast }\right) ^{\ast }\text{ and \ }%
x\rightarrow y=x^{\ast }\oplus y,
\end{equation*}

for every $x,y\in M,$ see [Mu; 07].

In\textbf{\ }[COM; 00], for an MV-algebra $\left( M,\oplus ,^{\ast
},0\right) $ and $x,y\in M$ \ is defined the following order relation:%
\begin{equation*}
x\leq y\text{ if and only if }x^{\ast }\oplus y=1.
\end{equation*}

We recall that, the natural order $\leq $ determines on $M$ a bounded
distributive lattice structure in which for $x,y\in M,$ the join $x\vee y$
and the meet $x\wedge y$ are given by:
\begin{eqnarray*}
x\vee y &=&(x\odot y^{\ast })\oplus y=(y\odot x^{\ast })\oplus x\text{ \ \ \
\ \ \ \ \ \ and } \\
x\wedge y &=&(x^{\ast }\vee y^{\ast })^{\ast }=x\odot (x^{\ast }\oplus
y)=y\odot (y^{\ast }\oplus x).\text{ }
\end{eqnarray*}

MV-algebras are particular cases of residuated lattices. If $\left( M,\oplus
,^{\ast },0\right) $ is an MV-algebra, then using the auxiliary operations
defined above, $(M,\wedge ,\vee ,\odot ,\rightarrow ,0,1)$ becomes a
residuated lattice.

Moreover, if $(L,\wedge ,\vee ,\odot ,\rightarrow ,0,1)$ is a residuated
lattice and for $x,y\in L$ we define $x\oplus y=x^{\ast }\rightarrow y$ \
(equivalent with $x\oplus y=(x^{\ast }\odot y^{\ast })^{\ast }),$ then $%
(L,\oplus ,^{\ast },0)$ is an $MV$ algebra iff
\begin{equation*}
(x\rightarrow y)\rightarrow y=(y\rightarrow x)\rightarrow x,
\end{equation*}%
for every $x,y\in L,$ see [T; 99].

In fact, MV-algebras are commutative bounded BCK-algebras, see [I; 09] and
[Me-Ju; 94].\medskip

\begin{definition}
\label{Definition 2.10}([COM; 00]) An algebra $\left( W,\rightarrow ,^{\ast
},1\right) $ of type $\left( 2,1,0\right) ~$is called a \textit{Wajsberg
algebra }if the following axioms are verified:

W1) $1\rightarrow x=x;$

W2) $\left( x\rightarrow y\right) \rightarrow \left[ \left( y\rightarrow
z\right) \rightarrow \left( x\rightarrow z\right) \right] =1;$

W3) $\left( x\rightarrow y\right) \rightarrow y=\left( y\rightarrow x\right)
\rightarrow x;$

W4) $\left( x^{\ast }\rightarrow y^{\ast }\right) \rightarrow \left(
y\rightarrow x\right) =1,$ for every $x,y,z\in W.$
\end{definition}

We remark that if $\left( W,\rightarrow ,^{\ast },1\right) $ is a Wajsberg
algebra and if on $W$ is defined the following binary relation
\begin{equation*}
x\leq y~\text{if~and~only~if~}x\rightarrow y=1,
\end{equation*}%
then, this relation is an order relation, called \textit{the natural order
relation on }$W\medskip $\textit{\ }(see [FRT; 84]).

\begin{remark}
\label{Remark 2.11}([COM; 00], Lemma 4.2.2 and Theorem 4.2.5)

i) If $\left( W,\rightarrow ,^{\ast },1\right) $ is a Wajsberg algebra,
defining
\begin{equation*}
x\odot y=\left( x\rightarrow y^{\ast }\right) ^{\ast }\text{ and \ }x\oplus
y=x^{\ast }\rightarrow y,
\end{equation*}%
for all $x,y\in W$, we obtain that $\left( W,\oplus ,\odot ,^{\ast
},0,1\right) $ is an MV-algebra.

ii) If $\left( M,\oplus ,\odot ,^{\ast },0,1\right) $ is an MV-algebra,
defining on $M$ \ the operation%
\begin{equation*}
x\rightarrow y=x^{\ast }\oplus y,
\end{equation*}%
it results that $\left( X,\rightarrow ,^{\ast },1\right) $ is a Wajsberg
algebra.
\end{remark}

\begin{remark}
\label{Remark 2.12}([I; 09]) A residuated lattice $(L,\wedge ,\vee ,\odot
,\rightarrow ,0,1)$ with Chang condition:%
\begin{equation*}
(C):\text{ }x\vee y=(x\rightarrow y)\rightarrow y
\end{equation*}%
for every $x,y\in L,$ is an equivalent definition of \ Wajsberg algebra.
\end{remark}

\section{Some remarks regarding the lattice of ideals in a commutative ring}

It is known that if\textbf{\ } $A$ is a commutative unitary ring, then $%
(Id(A),\cap ,+,\otimes \rightarrow ,0=\{0\},1=A)$ is a residuated lattice in
which the order relation is $\subseteq $ and $I\rightarrow J=(J:I),$ for
every $I,J\in Id(A)$, see [TT; 22].

Thus, since $\otimes $ and $\rightarrow $ form and adjoint pair, LR3)
becomes
\begin{equation*}
\ I\otimes J\subseteq K\text{ \ iff \ }I\subseteq (K:J),\text{ \ for every \
}I,J,K\in Id(A).
\end{equation*}

\begin{proposition}
\label{Proposition 3.1} \textit{\ Let }$A$\textit{\ be a commutative unitary
ring. \ Then }$(Id(A),\subseteq ,\rightarrow ,1=A)$\textit{\ is a bounded
BCK-algebra in which }$I\rightarrow J=(J:I),$\textit{\ for every }$I,J\in
Id(A).$
\end{proposition}

\begin{proof}
To prove BCK1), using Remark \ref{Remark 2.1}(8), since \ $(I:J)\otimes
J\subseteq I$ we deduce that \ $(I:J)\otimes J\otimes (K:I)\subseteq I$ $%
\otimes (K:I)\subseteq K.$ Since $(Id(A),\cap ,+,\otimes \rightarrow
,0=\{0\},1=A)$ is a residuated lattice we deduce that $(I:J)\otimes
(K:I)\subseteq (K:J)$ and finally, $(I:J)\subseteq ((K:J):(K:I)).$

Now, using LR3) and Remark \ref{Remark 2.1} (8) we \ have \ $(J:I)\otimes
I\subseteq J$ iff $I\subseteq (J:(J:I)),$ so BCK2) holds.

The condition BCK3) follows from Remark \ref{Remark 2.1} (9).
\end{proof}

\begin{proposition}
\label{Proposition 3.2} \textit{Let }$A$\textit{\ be a commutative unitary
ring and }$I,J\in Id(A).$\textit{\ Then:}

\textit{(i) }$\ \ (J:(J:(J:I)))=(J:I);$

\textit{(ii) }$\ \ (I:J)\subseteq (Ann(J):Ann(I));$

\textit{(iii) }$\ \ I\otimes Ann(I)=\mathbf{0};$

\textit{(iv) }$\ \ I\otimes J=\mathbf{0}$\textit{\ iff }$I\subseteq Ann(J);$

\textit{(v) }$\ \ \ Ann(I\otimes J)=(Ann(J):I)=(Ann(I):J);$

\textit{(vi) }$\ \ \ I$\textit{\ and }$J$\textit{\ coprime ideals implies }$%
I^{n}$\textit{\ and }$J^{n}$\textit{\ coprime, for every }$n\geq 1.$
\end{proposition}

\begin{proof}
Results from $i)-v)$ can be obtained by straightforward calculation.

$(vi).$ Suppose that $I$ and $J$ coprime ideals. Then $I+J=A.$

We recall that if $(L,\wedge ,\vee ,\odot ,\rightarrow ,0,1)$ is a
residuated lattice, then $x\vee y=1$ implies $x^{n}\vee y^{n}=1,$ for every $%
n\geq 1,$ see [BT; 03].

Since, $(Id(A),\cap ,+,\otimes \rightarrow ,0=\{0\},1=A)$ is a residuated
lattice, we deduce that $I^{n}$ and $J^{n}$ coprime ideals, for every $n\geq
1.$
\end{proof}

\medskip

It is known that, a\textbf{\ }commutative ring in which the lattice of
ideals (augmented with the ideal product) is isomorphic to an MV-algebra, is
a direct sum of local Artinian chain rings with unit, see [BN; 09].

\medskip

In the following, using the connections between some subvarieties of
residuated lattices, we give new characterizations for commutative and
unitary rings for which the lattice of ideals is an MV-algebra.

Using Remark \ref{Remark 2.12} and Proposition \ref{Proposition 3.1} we can
give the following definition:

\begin{definition}
\label{Definition 3.3} A commutative unitary ring $A$ has Chang property if:

\begin{equation*}
\text{ }I+J=(J:(J:I)),
\end{equation*}%
for every $I,J\in Id(A).$
\end{definition}

\begin{example}
\label{Example 3.28 copy(1)}  If we consider the commutative ring $(\mathbb{Z%
}_{4},+,\cdot ),$ then $Id(\mathbb{Z}_{4})=\{\{\widehat{0}\},\{\widehat{0},%
\widehat{2}\},$ $\mathbb{Z}_{4}\}$. We remark that
\begin{equation*}
\begin{tabular}{l|lll}
$\ \ +$ & $\{\widehat{0}\}$ & $\{\widehat{0},\widehat{2}\}$ & $\mathbb{Z}_{4}
$ \\ \hline
$\{\widehat{0}\}$ & $\{\widehat{0}\}$ & $\{\widehat{0},\widehat{2}\}$ & $%
\mathbb{Z}_{4}$ \\
$\{\widehat{0},\widehat{2}\}$ & $\{\widehat{0},\widehat{2}\}$ & $\mathbb{Z}%
_{4}$ & $\mathbb{Z}_{4}$ \\
$\mathbb{Z}_{4}$ & $\mathbb{Z}_{4}$ & $\mathbb{Z}_{4}$ & $\mathbb{Z}_{4}$%
\end{tabular}%
\end{equation*}%
and $I+J=(J:(J:I)),$ for  every $I,J\in Id(\mathbb{Z}_{4})$, so, the ring $(%
\mathbb{Z}_{4},+,\cdot )$ has Chang property.
\end{example}

Since, MV-algebras are commutative bounded BCK algebras, i.e., bounded
BCK-algebras that satisfy the additional condition $(x\rightarrow
y)\rightarrow y=(y\rightarrow x)\rightarrow x,$ see [COM; 00] and Wajsberg
algebras are termwise equivalent with MV-algebras (see Remark \ref{Remark
2.11}), we conclude that:

\begin{theorem}
\label{Theorem 3.4}\textit{A commutative ring has Chang property if and only
if its lattice of ideals is an MV-algebra.}
\end{theorem}

\begin{proposition}
\label{Proposition 3.5} \textit{\ Let }$(L,\vee ,\wedge ,\odot ,\rightarrow
,0,1)$\textit{\ be a residuated lattice. The following assertions are
equivalent:}

\textit{(i) }$\ (x\rightarrow y)\rightarrow y=(y\rightarrow x)\rightarrow x,$%
\textit{\ for every \thinspace }$x,y\in L;$

\textit{(ii) }$\ \ x\vee y=(x\rightarrow y)\rightarrow y,$ \textit{for every
\thinspace }$x,y\in L;$

$\bigskip $\textit{(iii) }$x^{\ast \ast }=x,$\textit{\ }$x\wedge y=x\odot
(x\rightarrow y)$\textit{\ and }$(x\rightarrow y)\vee (y\rightarrow x)=1,$%
\textit{\ for every \thinspace }$x,y\in L;$

\textit{(iv) }$\ ((x\rightarrow y)\rightarrow y)\rightarrow x=y\rightarrow
x, $\textit{\ for every \thinspace }$x,y\in L;$

\textit{(v) \ \ For \thinspace }$x,y\in L,$\textit{\ }$x\leq y$\textit{\
implies }$(y\rightarrow x)\rightarrow x\leq y.$
\end{proposition}

\begin{proof}
$(i)\Leftrightarrow (ii).$ Follows by [I; 09] and [T; 99].

$(i)\Leftrightarrow (iii).$ $\ $We recall that a BL-algebra is a residuated
lattice that verifies divisibility and prelinearity conditions: $x\wedge
y=x\odot (x\rightarrow y)$ and $(x\rightarrow y)\vee (y\rightarrow x)=1,$
for every \thinspace $x,y\in L.$ MV-algebras are BL-algebras which satisfies
double negation condition: $x^{\ast \ast }=x,$ for every \thinspace $x\in L,$
see [T; 99].

$(i)\Rightarrow (iv).$ Let $x,y\in L.$ Using (i), $L$ is an MV-algebra, so, $%
(x\rightarrow y)\rightarrow y=$ $(y\rightarrow x)\rightarrow x=x\vee y.$
Then $((x\rightarrow y)\rightarrow y)\rightarrow x=(x\vee y)\rightarrow x=$ $%
(x\rightarrow x)\wedge (y\rightarrow x)=y\rightarrow x.$

$(vi)\Rightarrow (i).$ By hypothesis, \ $y\rightarrow x\leq ((x\rightarrow
y)\rightarrow y)\rightarrow x,$ so, $[(x\rightarrow y)\rightarrow y]\odot $\
$(y\rightarrow x)\leq x$ and $(x\rightarrow y)\rightarrow y\leq
(y\rightarrow x)\rightarrow x.$ Similarly, $(y\rightarrow x)\rightarrow
x\leq (x\rightarrow y)\rightarrow y,$ thus, $(x\rightarrow y)\rightarrow
y=(y\rightarrow x)\rightarrow x,$\textit{\ for every \thinspace }$x,y\in L.$

$(i)\Rightarrow (v).$ Obviously.

$(v)\Rightarrow (i).$ Let $x,y\in L.$ From $x\leq (x\rightarrow
y)\rightarrow y,$ we deduce that $((x\rightarrow y)\rightarrow y)\rightarrow
x)\rightarrow x\leq (x\rightarrow y)\rightarrow y,$ that is, $%
(((x\rightarrow y)\rightarrow y)\rightarrow x)\rightarrow x=(x\rightarrow
y)\rightarrow y.$

From $y\leq (x\rightarrow y)\rightarrow y$ we deduce successively $%
((x\rightarrow y)\rightarrow y)\rightarrow x\leq y\rightarrow x,$ $%
(y\rightarrow x)\rightarrow x\leq (((x\rightarrow y)\rightarrow
y)\rightarrow x)\rightarrow x=(x\rightarrow y)\rightarrow y.$ So $%
(y\rightarrow x)\rightarrow x\leq (x\rightarrow y)\rightarrow y$ and
similarly $(x\rightarrow y)\rightarrow y\leq (y\rightarrow x)\rightarrow x.$
We conclude that $(x\rightarrow y)\rightarrow y=(y\rightarrow x)\rightarrow
x.$
\end{proof}

Using Theorem \ref{Theorem 3.4} and Proposition \ref{Proposition 3.5} we
deduce that:

\begin{corollary}
\label{Corollary 3.6}\textit{\ Let }$A$\textit{\ be a commutative unitary
ring. The following assertions are equivalent:}

\textit{(i) }$\ A$\textit{\ has Chang property;}

\textit{(ii) }$(Id(A),\cap ,+,Ann,0=\{0\},1=A)$\textit{\ is an MV-algebra;}

\textit{(iii) }$\ (I:(I:J))=(J:(J:I))$\textit{, for every }$I,J\in Id(A);$

\textit{(iv) }$\ Ann(Ann(I))=I,$\textit{\ }$\ I\cap J=(J:I)\otimes
I=(I:J)\otimes J$\textit{\ \ and }$(I:J)+(J:I)=A,$\textit{\ for every }$%
I,J\in Id(A);$

\textit{(v) }$\ \ (I:(J:(J:I)))=(I:J),$\textit{\ for every }$I,J\in Id(A);$

\textit{(vi) \ For }$I,J\in Id(A),$\textit{\ }$I\subseteq J$\textit{\
implies }$(I:(I:J))\subseteq J.$
\end{corollary}

Using some connections between divisible and involutive properties in
residuated lattices we give a new characterization for commutative rings
with Chang property:

\begin{proposition}
\label{Proposition 3.7}\textit{\ Let }$(L,\vee ,\wedge ,\odot ,\rightarrow
,0,1)$\textit{\ be a residuated lattice. Then }$\ L$\textit{\ is divisible
satisfying (DN) condition if and only if }$L$\textit{\ is an MV-algebra.}
\end{proposition}

\begin{proof}
Obviously, an MV-algebra is an involutive divisible residuated lattice.

Conversely, if $L$ is involutive and divisible, $x=x^{\ast \ast }$ and $%
x\vee y=(x\vee y)^{\ast \ast }=(x^{\ast }\wedge y^{\ast })^{\ast }=[$ $%
x^{\ast }\odot (x^{\ast }\rightarrow y^{\ast })]^{\ast }=$ $(x^{\ast
}\rightarrow y^{\ast })$ $\rightarrow x^{\ast \ast }=(x^{\ast }\rightarrow
y^{\ast })\rightarrow x=(y\rightarrow x)\rightarrow x$, for every $x,y\in L.$
Thus, $L$ is an MV-algebra.
\end{proof}

Using Proposition \ref{Proposition 3.7} we deduce that:

\begin{corollary}
\label{Corollary 3.8}\textit{Let }$A$\textit{\ be a commutative unitary
ring. The following assertions are equivalent:}

\textit{(i) }$\ A$\textit{\ has Chang property;}

\textit{(ii) }$I\cap J=(J:I)\otimes I$\textit{\ and }$Ann(Ann(I))=I,$\textit{%
\ for every }$I,J\in Id(A).$
\end{corollary}

\begin{proposition}
\label{Proposition 3.9}\textit{Let} $A$ \textit{be a commutative ring and} $%
I,J\in Id\left( A\right) $. \textit{The following relations are true:}

1) $I+J=<I\cup J>\subseteq \left( I:\left( I:J\right) \right) ,$ for every $%
I,J\in Id\left( A\right) $;

2) \textit{If} $A$ \textit{is a principal ideal domain, then} $\left(
I:\left( I:J\right) \right) =\left( J:\left( J:I\right) \right) =I+J$,
\textit{for} $I$ \textit{and} $J$ \textit{nonzero ideals};

3) \textit{If} $\mathcal{A}$ \textit{is a ring factor of a principal ideal
domain }$D$\textit{, then} $\left( \mathcal{I}:\left( \mathcal{I}:\mathcal{J}%
\right) \right) =\left( \mathcal{J}:\left( \mathcal{J}:\mathcal{I}\right)
\right) =\mathcal{I}+\mathcal{J}$\textit{, \ for} $\mathcal{I}$ \textit{and}
$\mathcal{J}$ \textit{arbitrary ideals in }$\mathcal{A}$.
\end{proposition}

\begin{proof}
1) Let $x\in I+J$, therefore $x=ai+bj,a,b\in A$ and $i\in I$, $j\in J$. We
have that $\left( I:J\right) =\{y\in A:y\cdot J\subseteq I\}$. For $y\in
\left( I:J\right) $ we have $xy=\left( ai+bj\right) y=aiy+bjy$, with $aiy\in
I$ and $bjy\in I$. It results that, $xy\in I$.

2) If $A$ is a principal ideal domain, then let $a,b\in A$ be two nonzero
elements and $I=<a>$, $J=<b>$ be the principal ideals generated by $a$ and $%
b $.

Let $d=$\textit{gcd}$\{a,b\}$. We have $d=a\gamma +b\beta ,$ $a,b\in A$, $%
a=a_{1}d$ and $b=b_{1}d$, with $1=$\textit{gcd}$\{a_{1},b_{1}\}$.

We will prove that $\left( I:\left( I:J\right) \right) =\left( J:\left(
J:I\right) \right) =<d>=I+J$. First of all, we remark that $\left(
I:J\right) =$ $<a_{1}>$. Indeed, if $y\in \left( I:J\right) $, we have $%
yb_{1}d\in I$, that means $yb_{1}d=$ $a_{1}dq,q\in A$ and $yb_{1}=$ $a_{1}q$%
. Since $1=$\textit{gcd}$\{a_{1},b_{1}\}$, it results $a_{1}\mid y$ and $%
y\in <a_{1}>$. If $y\in <a_{1}>$, we have $y=a_{1}z$ and $yb=a_{1}zb_{1}d\in
I$. By repeating the same proof, since $a_{1}=$\textit{gcd}$\{a,a_{1}\}$, we
obtain that $\left( I:\left( I:J\right) \right) =<d>=I+J$. \ In the same
way, the equality $\left( J:\left( J:I\right) \right) =<d>=I+J$ can be
obtained.

If $I=\{0\}$, since $A$ is an integral domain, we have that $\left( \mathbf{0%
}:\left( \mathbf{0}:J\right) \right) =Ann(Ann(J))=A,$ for every $J\in
Id\left( A\right) \backslash \{0\}$ and $\left( J:\left( J:\mathbf{0}\right)
\right) =J.$

We deduce that if $A$ is a principal ideal domain, then $Id(A)$ is only a
noncommutative BCK-algebra.

3) We apply 2) from above.

Let $D$ be a principal ideal domain and $W=\alpha D,$ $\mathbb{\alpha \in }\
D$, be an ideal in $D$. First of all, we remark that if $D\,$\ is an
integral domain, its factor ring is not allways an integral domain. For
example, $\mathbb{Z}$ and $\mathbb{Z}_{n}$. Let $\mathcal{A}=D/W$ be the
factor ring, which is a principal ring. Let $\mathcal{I},\mathcal{J}$ be two
nonzero ideals of $\mathcal{A}$, $\mathcal{I}=I/W$ and $\mathcal{J}=J/W$,
where $I$ and $J$ are nonzero ideals in $D$, with $W\subseteq I$ and $%
W\subseteq J$. It results that $I=\beta W$ and $J=\gamma W$, $\beta \mid
\alpha $ and $\gamma \mid \alpha $, therefore $\mathcal{I}=<\widehat{\beta }%
> $ and $\mathcal{J}=<\widehat{\gamma }>$, $\beta $, $\gamma \in D$. We will
prove that $\left( \mathcal{I}:\left( \mathcal{I}:\mathcal{J}\right) \right)
=\left( \mathcal{J}:\left( \mathcal{J}:\mathcal{I}\right) \right) =<\widehat{%
\delta }>=\mathcal{I}+\mathcal{J}$, where $\delta =$\textit{gcd}$\{\beta
,\gamma \}$. From the above, we have that $\mathcal{I}+\mathcal{J}\subseteq
\left( \mathcal{I}:\left( \mathcal{I}:\mathcal{J}\right) \right) $ and $%
\mathcal{I}+\mathcal{J}\subseteq \left( \mathcal{J}:\left( \mathcal{J}:%
\mathcal{I}\right) \right) $. Let $\delta =m\beta +n\gamma ,\beta =\delta
\beta _{1}$ and $\gamma =\delta \gamma _{1}$ with $1=$\textit{gcd}$\{\beta
_{1},\gamma _{1}\}$. From 2), we have $\left( \mathcal{I}:\mathcal{J}\right)
=<\widehat{\beta _{1}}>$. If $\widehat{y}\in <\widehat{\beta _{1}}>$, we
have $\widehat{y}=\widehat{\beta _{1}}\widehat{z}$ and $\widehat{y}\widehat{%
\beta }=\widehat{\beta _{1}}\widehat{z}\widehat{\gamma _{1}}\widehat{\delta }%
\in \mathcal{I}$. By repeating the same proof, since $\beta _{1}=$\textit{gcd%
}$\{\beta ,\beta _{1}\}$, we obtain that $\left( \mathcal{I}:\left( \mathcal{%
I}:\mathcal{J}\right) \right) =<\widehat{\delta }>=\mathcal{I}+\mathcal{J}$.
\ In the same way, the equality $\left( \mathcal{J}:\left( \mathcal{J}:%
\mathcal{I}\right) \right) =<\widehat{\delta }>=\mathcal{I}+\mathcal{J}$ can
be obtained.

If $I=\{0\}$ and $J$ is a nonzero ideal in $D,$ the equality is also true.
We must prove that $\left( \mathbf{0}:\left( \mathbf{0}:\mathcal{J}\right)
\right) =\left( \mathcal{J}:\left( \mathcal{J}:\mathbf{0}\right) \right) $,
with $\mathbf{0}$ zero ideal in $\mathcal{A}$, which is equivalent to $%
Ann\left( Ann\left( \mathcal{J}\right) \right) =\mathcal{J}$. For $\mathcal{J%
}=<\widehat{\gamma }>$, we have $Ann\left( \mathcal{J}\right) =<\widehat{%
\sigma }>$, with $\alpha =\gamma \sigma $, therefore $Ann\left( Ann\left(
\mathcal{J}\right) \right) =<\widehat{\gamma }>=\mathcal{J}$.
\end{proof}

From \ Corollary \ref{Corollary 3.6} and Proposition \ref{Proposition 3.9}
(3) we deduce that:

\begin{theorem}
\label{Theorem 3.10}\textit{A ring factor of a principal ideal domain has
Chang property. }
\end{theorem}

\begin{example}
\label{Example 3.15}

1) The ring of integers $(\mathbb{Z},\mathbb{+},\cdot )$ does not have Chang
property \ because the relation $\left( I:\left( I:J\right) \right) =\left(
J:\left( J:I\right) \right) $, is not true for $I=\mathbf{0}$. \ Indeed,
since $\mathbb{Z}$ is principal ideal domain, we have $\left( \mathbf{0}%
:\left( \mathbf{0}:J\right) \right) =Ann\left( Ann\left( J\right) \right) =%
\mathbb{Z}$ and $\left( J:\left( J:\mathbf{0}\right) \right) =J$, for every $%
J\in Id\left( \mathbb{Z}\right) .$

2) Let $K$ be a field, $K\left[ X\right] $ the polynomial ring and $f\in K%
\left[ X\right] $. Therefore, from Proposition \ref{Proposition 3.9}, the
lattice of ideals of the quotient ring $A=K\left[ X\right] /\left( f\right) $
is an MV-algebra and has Chang property.

Let $K=\mathbb{R}$ and $f\left( X\right) =\left( X+1\right) \left(
X+2\right) \left( X+3\right) $. Therefore, $A=\mathbb{R}[X]/\left( f\right)
\simeq \mathbb{R}[X]/\left( X+1\right) \times \mathbb{R}[X]/\left(
X+2\right) \times \mathbb{R}[X]/\left( X+3\right) $, from Chinese Remainder
Theorem. From here, we have that $A\simeq \mathbb{R}\times \mathbb{R}\times
\mathbb{R}$. In this case, $Id\left( A\right) $ is a Boolean lattice. Indeed
$Id\left( A\right) =\{0$, $0\times 0\times \mathbb{R}$, $0\times \mathbb{R}%
\times 0$, $\mathbb{R}\times 0\times 0$, $0\times \mathbb{R}\times \mathbb{R}
$, $\mathbb{R}\times 0\times \mathbb{R}$, $\mathbb{R}\times \mathbb{R}\times
0$, $\mathbb{R}\times \mathbb{R}\times \mathbb{R}\}$ and is a Boolean
algebra.

Let $K=\mathbb{R}$ and $f\left( X\right) =\left( X+1\right) ^{2}\left(
X+2\right) $. Therefore, $A=\mathbb{R}[X]/\left( f\right) \simeq \mathbb{R}%
[X]/\left( X+1\right) ^{2}\times \mathbb{R}[X]/\left( X+2\right) \simeq
\mathbb{R}[X]/\left( X+1\right) ^{2}\times \mathbb{R}$, from Chinese
Remainder Theorem. In this case $Id\left( A\right) $ is an MV-algebra.
Indeed, $I=\mathbb{R}\times 0$ is an ideal in $A$, $Ann\left( I\right) =$ $%
0\times \mathbb{R}$. In this case $I$ and $Ann\left( I\right) $ are not
coprime ideals and $Id\left( A\right) $ is not a Boolean algebra is an
MV-algebra.

\textbf{The result can be generalized:} If $\mathbb{K}$ is a field and $%
f=f_{1}^{k_{1}}f_{2}^{k_{2}}...f_{r}^{k_{r}},$ with $f_{i}$ irreducible
polynomials, we have that $A=\mathbb{K}[X]/\left( f\right) \simeq \prod
\mathbb{K}[X]/\left( f_{i}^{k_{i}}\right) $, therefore the algebra $Id\left(
A\right) $ is a Boolean algebra or an MV-algebra (which is not Boolean).
\end{example}

\begin{remark}
Using Example \ref{Example 3.15} (2) we remark that the condition from
Theorem \ref{Theorem 3.10} is a necessary condition for a ring to have Chang
property.
\end{remark}

We recall the fundamental theorem of finite abelian groups:

\begin{theorem}
\label{Theorem 3.12}[HE; 75, Theorem 2.14.1] \textit{Every finite abelian
group is a direct product of cyclic groups of prime power order. Moreover,
the number of terms in the product and the orders of the cyclic groups are
uniquely determined by the group.}
\end{theorem}

\bigskip

It is known that the lattice of ideals of a commutative unitary ring is a
modular lattice (with respect to set inclusion), see [Bl; 53].

\begin{corollary}
\label{Corollary 3.13} \textit{\ Any finite commutative unitary ring has
Chang property.}
\end{corollary}

\begin{proof}
Let $A$ be a finite commutative unitary ring with $\left\vert A\right\vert
=n=p_{1}^{\alpha _{1}}\cdot ...\cdot p_{r}^{\alpha _{r}}.$ We prove that the
set of all ideals of the rings $A$ forms an MV-algebra $\left( Id\left(
A\right) ,\vee ,\wedge ,\odot ,\rightarrow ,0,1\right) $ in which the order
relation is $\subseteq $, $I\odot J=I\otimes J$, $I^{\ast }=Ann(I)$, $%
I\rightarrow J=\left( J:I\right) $, $I\vee J=I+J$, $I\wedge J=I\cap J$, $%
0=\{0\}$ and $1=A$.\medskip\

From Theorem \ref{Theorem 3.12}, we know that $A$ has the form $A=\mathbb{Z}%
_{k_{1}}\times \mathbb{Z}_{k_{2}}\times ...\times \mathbb{Z}_{k_{r}},$ where
$k_{i}=p_{i}^{\alpha _{i}}$ and $p_{i}$ a prime number, for all $i\in
\{1,2,...,r\}$. \

It results that $Id\left( A\right) =$ $Id(\mathbb{Z}_{k_{1}})\times Id(%
\mathbb{Z}_{k_{2}})\times ...\times Id(\mathbb{Z}_{k_{r}})$. \ Now, we apply
Theorem \ref{Theorem 3.10}.

We recall that in an MV-algebra $(A,\oplus ,\odot ,^{\ast },0,1)$ \ we have $%
x\oplus y=(x^{\ast }\odot y^{\ast })^{\ast }=x^{\ast }\rightarrow y=y^{\ast
}\rightarrow x$ and $x\rightarrow y=x^{\ast }\oplus y=(x\odot y^{\ast
})^{\ast },$ for every $x,y\in A,$ see [COM; 00] and [MU; 07]. Using these
equalities, in MV-algebra \ $Id\left( A\right) $ we obtain:
\begin{equation*}
I\oplus J=Ann(Ann(I)\otimes Ann(J))=(J:Ann(I))=(I:Ann(J))
\end{equation*}

and%
\begin{equation*}
I\rightarrow J=(J:I)=Ann(I)\oplus J=Ann(I\otimes Ann(J)),
\end{equation*}

for every $I,J\in Id\left( A\right) $, since, in this case, $Ann(Ann(I))=I.$
\end{proof}

\begin{remark}
\label{Remark 3.14} If $A$ is a finite commutative ring with $\left\vert
A\right\vert =n=p_{1}^{\alpha _{1}}\cdot ...\cdot p_{r}^{\alpha _{r}},$ then
\textit{the set} $Id\left( A\right) $ \textit{has} $\mathcal{N}_{A}=\overset{%
r}{\underset{i=1}{\prod }}\left( \alpha _{i}+1\right) $ \textit{elements}.
\end{remark}

\begin{remark}
\label{Remark 3.16} \ If the commutative ring $A$ is not finite, then $Id(A)$
is not necessary an MV-algebra. For example, $Id(\mathbb{Z},+,\cdot )$ is
only a noncommutative BCK-algebra, see Example \ref{Example 3.15}.\newline
\end{remark}

The following results can be obtained by the straightforward calculations.

\begin{proposition}
\label{Proposition 3.17} \textit{Let }$A$\textit{\ be a finite commutative
unitary ring. Then:}

(i) $\ (Id(A),\cap ,+,0=\{0\},1=A)$\textit{\ is a bounded distributive
lattice in which }%
\begin{equation*}
I\vee J=I+J=(I:(I:J))=(J:(J:I)),
\end{equation*}%
\begin{equation*}
I\wedge J=I\cap J=(J:I)\otimes I=(I:J)\otimes J,
\end{equation*}%
\textit{for every }$I,J\in Id(A).$

(ii) $Ann(Ann(I))=I,$\textit{\ for every }$I\in Id(A);$

(iii) $I\oplus Ann(I)=A,$\textit{\ for every }$I\in Id(A);$

(iv) $Ann(I\cap J)=Ann(I)+Ann(J),$\textit{\ for every }$I,J\in Id(A).$
\end{proposition}

We recall that a residuated lattice $(L,\vee ,\wedge ,\odot ,\rightarrow
,0,1)$ in which $x^{2}=x,$ for all $x\in L,$ is called a \emph{Heyting
algebra (or\ G(RL)\ algebra or pseudo Boolean algebra}), see [I; 09], [GDCK;
10], [T; 99].

In [BNM; 10] was proved that unitary commutative rings for which the
semiring of ideals, under ideal sum and ideal product, are Heyting algebras
are exactly Von Neumann regular rings, i.e. commutative rings $A$ in which
for every element $x\in A$ there exists an element $a\in A$ such that $%
x=a\cdot x^{2}.$

\medskip

In the following, we give new characterizations for commutative rings in
which the lattice of ideals is a Heyting algebra.

\begin{proposition}
\label{Proposition 3.18.1} \textit{Let }$(L,\vee ,\wedge ,\odot ,\rightarrow
,0,1)$\textit{\ be a residuated lattice. The following assertions are
equivalent:}

(i) for $x,y,z\in L,$ if $x\odot y\leq z$ and $x\leq y,$ then $x\leq z;$

(ii) for every $z\in L,$ $D_{z}=\{x\in L:z\leq x\}$ is a deductive system of
$L;$

(iii) $x^{2}\leq y$ implies $x\leq y,$ \ for every $x,y\in L;$

(iv) $x\odot y\leq z$ implies $x\odot (x\rightarrow y)\leq z,$ for every $%
x,y,z\in L;$

(v) $x=x^{2},$ \ for every $x\in L.$
\end{proposition}

\begin{proof}
$(i)\Rightarrow (ii).$ Since for every $z\in L,z\leq 1,$ we have that $1\in
D_{z}.$ If $x,x\rightarrow y\in D_{z},$ then $z\leq x,x\rightarrow y.$ Using
(i), $z\leq y,$ hence $y\in D_{z},$ that is, $D_{z}$ is a deductive system
of $L.$

$(ii)\Rightarrow (i).$ Let $x,y,z\in L$ such that $x\odot y\leq z$ and $%
x\leq y.$ Then $y,y\rightarrow z\in D_{x}.$ Since $D_{x}$ is a deductive
system of $L$, we obtain $z\in D_{x}$ so, $x\leq z\mathbf{.}$

$(ii)\Rightarrow (iii).$ Let $x,y\in L$ such that $x^{2}\leq y.$ Then $%
x\rightarrow y\in D_{x}.$ Since $x,x\rightarrow y\in D_{x},$ we deduce that $%
y\in D_{x},$ that is, $x\leq y.$

$(iii)\Rightarrow (iv).$ Let $x,y,z\in L$ such that $x\odot y\leq z$ $.$
Since $y\rightarrow z\leq (x\rightarrow y)\rightarrow (x\rightarrow z),$ we
obtain $1=(x\odot y)\rightarrow z$ $=x\rightarrow (y\rightarrow z)\leq
x\rightarrow \lbrack (x\rightarrow y)\rightarrow (x\rightarrow z)],$ hence $%
x\leq (x\rightarrow y)\rightarrow (x\rightarrow z).$ Finally, we obtain that
$x\leq x\rightarrow \lbrack (x\rightarrow y)\rightarrow z].$

By hypothesis, we deduce that $x\leq (x\rightarrow y)\rightarrow z,$ hence $%
x\odot (x\rightarrow y)\leq z.$

$(iv)\Rightarrow (iii).$ Consider $y,z\in L$ such that $y^{2}\leq z.$ By
hypothesis, $y\odot (y\rightarrow y)\leq z,$ that is, $y\leq z.$

$(iii)\Rightarrow (i).$ Let $x,y,z\in L$ such that $x\odot y\leq z$ and $%
x\leq y.$

We have $1=x\rightarrow (y\rightarrow z)=y\rightarrow (x\rightarrow z)$ $%
\leq (x\rightarrow y)\rightarrow (x\rightarrow (x\rightarrow z)),$ hence $%
x\rightarrow y\leq x\rightarrow (x\rightarrow z).$ Then $x\leq x\rightarrow
z,$ so, we obtain that $x\leq z\mathbf{.}$

$(iii)\Rightarrow (v).$ We have $x^{2}\leq x^{2},$ hence $x\leq x^{2}.$ We
deduce that $x=x^{2},$ \ for every $x\in L.$

$(v)\Rightarrow (iii).$ Let $x,y\in L$ such that $x^{2}\leq y.$ Since $%
x=x^{2},$ we deduce that $x\leq y.$
\end{proof}

\begin{theorem}
\label{Theorem 3.18}\textit{\ Let }$A$\textit{\ be a commutative ring. The
following conditions are equivalent:}

\textit{(i) }$A$\textit{\ is a Von Neumann regular ring;}

\textit{(ii) }$(J:I)\otimes I=I\otimes J=I\cap J,$\textit{\ for every }$%
I,J\in Id(A);$

\textit{(iii) } for \textit{\ }$I,J,K\in Id(A),$ if $I\otimes J\subseteq K$
and $I\subseteq J,$ then $I\subseteq K;$

\textit{(iv) }for every $K\in Id(A),$ $D_{K}=\{I\in Id(A):K\subseteq I\}$ is
a deductive system of the residuated lattice $(Id(A),\cap ,+,\otimes
\rightarrow ,0=\{0\},1=A);$

\textit{(v) }$I\otimes I\subseteq J$ implies $I\subseteq J,$ \ for every $%
I,J\in Id(A);$

\textit{(vi) }$I\otimes J\subseteq K$ implies $I\otimes (J:I)\subseteq K,$
for every \textit{\ }$I,J,K\in Id(A).$
\end{theorem}

\begin{proof}
$\ (i)\Rightarrow (ii).$ Let $I,J\in Id(A).$ Using $LR_{3}$ and the
properties of residuated lattices (see [BT; 03]) we obtain:
\begin{equation*}
(J:I)\otimes I\subseteq ((I\otimes J):(I\otimes I))\Leftrightarrow
(J:I)\otimes I\subseteq ((I\otimes J):I)\Leftrightarrow
\end{equation*}%
\begin{equation*}
(J:I)\subseteq (((I\otimes J):I):I)=((I\otimes J):(I\otimes I))=((I\otimes
J):I).
\end{equation*}%
We deduce that $(J:I)\otimes I\subseteq I\otimes J.$ Since $J\subseteq (J:I),
$ then $I\otimes J\subseteq (J:I)\otimes I,$ so $I\otimes J=(J:I)\otimes I.$
Clearly, $I\otimes J\subseteq I,J.$ To prove that $I\otimes J=I\cap J,$ let $%
K\in Id(A)$ such that $K\subseteq I,J.$ Then $K=K^{2}\subseteq I\otimes J,$
that is, $I\otimes J=I\cap J.$ $\ \ \ \ \ \ \ \ \ \ $

$(ii)\Rightarrow (i).$ For $I=J$ we obtain $I\otimes I=I\cap I=I,$ that is, $%
I\otimes I=I,$ for every $I\in Id(A).$

$\ (i)\Longleftrightarrow (iii)\Longleftrightarrow (iv)\Longleftrightarrow
(v)\Longleftrightarrow (vi).$ Follows by Proposition \ref{Proposition 3.18.1}%
.
\end{proof}

\begin{remark}
\label{Remark 3.19}\textit{\ \ Boolean rings are Von Neumann regular rings.
Indeed, since in a Boolean ring $A$, every ideal is idempotent, using
Theorem \ref{Theorem 3.18}, we obtain that} $Id(A)$\textit{\ \ is a Heyting
algebra. }
\end{remark}

\medskip

We recall the following result that characterizes unitary rings for which
the lattice of ideals is a Boolean algebra:

\begin{theorem}
\label{Theorem 3.20}\textit{[Bl; 53], [CL; 19] \ For a ring }$A,$\textit{\
the following assertions are equivalent:}

\textit{(i) The lattice }$Id(A)$\textit{\ is a Boolean algebra;}

\textit{(ii) }$A$\textit{\ is isomorphic to a finite direct sum of division
rings;}

\textit{(iii) }$A$\textit{\ is isomorphic to a finite direct sum of simple
rings;}

\textit{(iv) Ideals form both a Heyting algebra and an ortholattice.}
\end{theorem}

\medskip

It is known that a residuated lattice $(L,\vee ,\wedge ,\odot ,\rightarrow
,0,1)$ is a Boolean algebra if and only if $x\vee x^{\ast }=1,$ for every $%
x\in L,$ see [GDCK; 10].

\medskip

We give new necessary and sufficient conditions for rings whose lattice of
ideals is a Boolean algebra:

\begin{theorem}
\label{Theorem 3.21}\textit{\ Let }$A$\textit{\ be a commutative ring and }$%
I\in Id(A).$\textit{\ The following assertions are equivalent:}

\textit{(i) }$I$\textit{\ and }$Ann(I)$\textit{\ are coprime;}

\textit{(ii ) }$I\otimes I=I,$\textit{\ }$Ann(Ann(I))=I,$\textit{\ }$%
(Ann(I):I)+(I:Ann(I))=A$\textit{\ and \ }$Ann(I)\otimes (I:Ann(I))=0.$
\end{theorem}

\begin{proof}
$\ (i)\Rightarrow (ii).$ If $I+Ann(I)=A,$ using Remark \ref{Remark 2.3} (1),
we deduce that $I\otimes I=I\cap I=I.$ Moreover, $I$ is a Boolean element in
the residuated lattice $(Id(A),\cap ,+,\otimes \rightarrow ,0=\{0\},1=A)$,
so $Ann(Ann(I))=I$ and $I\cap Ann(I)=\mathbf{0}.$ Since $Ann(I)\otimes
(I:Ann(I))\subseteq I\cap Ann(I)=\mathbf{0}$ we deduce that $Ann(I)\otimes
(I:Ann(I))=\mathbf{0}.$

Also, $Ann(I)\subseteq (Ann(I):I)$ and $I\subseteq (I:Ann(I)).$ Then $%
A=Ann(I)+I\subseteq (Ann(I):I)+(I:Ann(I)),$ so $(Ann(I):I)+(I:Ann(I))=A.$

$(ii)\Rightarrow (i).$ From $I\otimes I=I=I,$ we have $(Ann(I):I)=Ann(I).$
Since $Ann(Ann(I))=I,$ and $(I:Ann(I))\otimes Ann(I)=\mathbf{0},$ we deduce
that $(I:Ann(I))=I.$ Then $Ann(I)+I=(Ann(I):I)+(I:Ann(I))=A.$ We conclude
that $I$ and $Ann(I)$ are coprime$.$
\end{proof}

\begin{proposition}
\label{Proposition 3.29} \textit{Let} $A$ \textit{be a commutative unitary
ring,} $x\in A,$ \textit{an idempotent element and} $I=<x>$ \textit{be the
ideal generated by }$x$. \textit{Therefore,} $Ann\left( I\right) $ \textit{%
and} $I$ \textit{are coprime ideals.}
\end{proposition}

\begin{proof}
If $x\not=1$, since $x+1-x=1$, we have that $1-x\not\in I$. Also, since $%
\left( 1-x\right) x=0,$ we have that $\left( 1-x\right) y=0$, for all $y\in
I $, then $1-x\in Ann\left( I\right) $ and $Ann\left( I\right) $ and $I$ are
coprime. If $x=1$, then $I=A$ and $Ann\left( I\right) =\{0\}$.
\end{proof}

\begin{theorem}
\label{Theorem 3.22}\textit{Let }$A$\textit{\ be a finite commutative ring
and }$I\in Id(A).$\textit{\ Then the following assertions are equivalent:}

\textit{(i) }$I$\textit{\ and }$Ann(I)$\textit{\ are coprime;}

\textit{(ii) }$I\otimes I=I.$
\end{theorem}

\begin{proof}
$\ (i)\Rightarrow (ii).$ Obviously, by Theorem \ref{Theorem 3.21}.

$\ (ii)\Rightarrow (i).$ Follows by Corollary \ref{Corollary 3.13} and
Theorem \ref{Theorem 3.21}, since in an MV-algebra $(L,\oplus ,^{\ast },0),$
$x^{\ast \ast }=x$ and $(x\rightarrow y)\vee (y\rightarrow x)=1,$ for every $%
x,y\in L.$ Also, if $x\in L$ such that $x\odot x=x,$ we deduce that
$x$ is a Boolean element in MV-algebra $L,$ so, $x^{\ast }$ is also Boolean.
Thus, $x^{\ast }\odot x^{\ast }=x^{\ast }.$

Then $x^{\ast }\odot (x^{\ast }\rightarrow x)=$ $x^{\ast }\odot (x^{\ast
}\rightarrow x^{\ast \ast })=x^{\ast }\odot (x^{\ast }\odot x^{\ast })^{\ast
}=x^{\ast }\odot x^{\ast \ast }=x^{\ast }\odot x=0.$ We deduce that $x^{\ast
}\odot (x^{\ast }\rightarrow x)=0.$

We conclude that in MV-algebra $(Id(A),\oplus ,Ann,0=\{0\})$, for every $%
I\in Id(A),$ if $I\otimes I=I$ \ then $Ann(Ann(I))=I,$ $%
(Ann(I):I)+(I:Ann(I))=A$ and \ $Ann(I)\otimes (I:Ann(I))=\mathbf{0}.$
\end{proof}

\begin{remark}
If $I$ is a non-idempotent ideal in \textit{a finite commutative ring }$A,$%
\textit{\ then }$I$\textit{\ and }$Ann(I)$\ are not coprime. Indeed, we
consider the ring $(\mathbb{Z}_{4},+,\cdot )$ \textit{\  and }$I=\{\widehat{0%
},\widehat{2}\}.$ Then $I\otimes I=\{\widehat{0}\}$\textit{\ and using
Example \ref{Example 3.28} (1), } $I$ and $Ann\left( I\right) $ are not
coprime ideals.
\end{remark}

The following results can be obtained by the straightforward calculations.

\begin{corollary}
\label{Corollary 3.23}\textit{\ Let }$A$\textit{\ be a finite commutative
ring. The following conditions are equivalent:}

\textit{(i) }$(Id(A),\cap ,+,Ann,0=\{0\},1=A)$\textit{\ is a Boolean algebra;%
}

\textit{(ii) }$I\otimes I=I,$\textit{\ for every }$I\in Id(A);$

\textit{(iii) }$I$\textit{\ and }$Ann(I)$\textit{\ are coprime, for every }$%
I\in Id(A);$

\textit{(iv) }$I\cap Ann(I)=0,$\textit{\ for every }$I\in Id(A);$

\textit{(v) }$I\oplus I=I,$\textit{\ for every }$I\in Id(A).$
\end{corollary}

Moreover, using the characterization of Boolean elements in MV-algebras we
obtain that a residuated lattice is a Boolean algebra if and only if it is
both Heyting algebra and MV-algebra.

Thus, we conclude that:

\begin{corollary}
\label{Corollary 3.25}\textit{\ Let }$A$\textit{\ be a commutative ring.
Then:}

\textit{(i) }$(Id(A),\cap ,+,Ann,0=\{0\},1=A)$\textit{\ is a Boolean algebra
if and only if }$A$ is a Von Neumann regular ring satisfying Chang property;

\textit{(ii) If }$A$\textit{\ is finite, then }$(Id(A),\cap
,+,Ann,0=\{0\},1=A)$\textit{\ is a Boolean algebra if and only if }$A$ is a
Von Neumann regular ring.
\end{corollary}

Using Remark \ref{Remark 3.19} and Corollary \ref{Corollary 3.23} we deduce
that:

\begin{proposition}
\label{Proposition 3.24} If $A$ is a finite Boolean ring, then $(Id(A),\cap
,+,Ann,0=\{0\},1=A)$ is a Boolean algebra.
\end{proposition}

\begin{proposition}
\label{Proposition 3.26} \textit{Let} \thinspace $p_{1},p_{2},...p_{r},$
\textit{be} $r$ \textit{prime but not necessarily distinct numbers}. \textit{%
We consider the ring} $A=\mathbb{Z}_{p_{1}}\times \mathbb{Z}_{p_{2}}\times
...\times \mathbb{Z}_{p_{r}}.$ \textit{Therefore, for each} $I\in Id\left(
A\right) ,$ \textit{we have that} $I$ \textit{and} $Ann\left( I\right) $
\textit{are coprime. The number of ideals in this ring is }$2^{r}$.
\end{proposition}

\begin{proof}
First of all, we remark that if $K$ is a field, then the only ideals are $%
\{0\}$ and $K.$ Then, $Ann(\{0\})=K$, $Ann\{K\}=\{0\}$ and it is clear that $%
Ann(\{0\})$ and $K$ are coprime.

\textbf{Case 1. }We assume that $p_{1},p_{2},...,p_{r}$ are distinct prime
numbers. For an arbitrary integer $q,q\in \{1,2,...,r\},$ let $%
v=p_{i_{1}}p_{i_{2}}...p_{i_{q}}$ and $w=p_{i_{q+1}}p_{i_{q+2}}...p_{i_{r}}$%
, where $%
\{p_{i_{1}},p_{i_{2}},...p_{i_{q}},p_{i_{q+1}},p_{i_{q+2}},...p_{i_{r}}\}$
and $\{p_{1},p_{2},...,p_{r}\}$ are the same set. Let $%
s=p_{1}p_{2}...p_{r}=p_{i_{1}}p_{i_{2}},...p_{i_{q}}p_{i_{q+1}}p_{i_{q+2}}...p_{i_{r}}.
$

We have that $A=\mathbb{Z}_{p_{1}}\times \mathbb{Z}_{p_{2}}\times ...\times
\mathbb{Z}_{p_{r}}\simeq \mathbb{Z}_{s}=\mathbb{Z}_{vw}$. Let $I$ be an
ideal of the ring $A\,$. Then $I$ has the form $I\simeq \mathbb{Z}_{v}\simeq
\mathbb{Z}_{p_{i_{1}}}\times \mathbb{Z}_{p_{i_{2}}}\times ...\times \mathbb{Z%
}_{p_{i_{q}}}=\{\widehat{0},\widehat{w},\widehat{2w},\widehat{3w},...,%
\widehat{\left( v-1\right) w}\},$ where $\widehat{x}$ is an element from $A$%
. It results that $Ann\left( I\right) =\{\widehat{0},\widehat{v},\widehat{2v}%
,\widehat{3v},...,\widehat{\left( w-1\right) v}\}$. Since $v$ and $w$ are
coprime integers, there are $a,b\in \mathbb{Z}$ such that $1=av+bw$.
Therefore, we found the elements $\widehat{av}\in Ann\left( I\right) $ and $%
\widehat{bw}\in I$ such that $\widehat{1}=$ $\widehat{av}+\widehat{bw}$.
From here, we get that $I$ and $Ann\left( I\right) $ are coprime ideals.

\textbf{Case 2.} The prime numbers $p_{1},p_{2},...,p_{r}$ are not distinct.
Without losing generality, we suppose that there is a number $s\leq r$ such
that $p_{1}=p_{2}=...=p_{s}=p$ and $p_{s+1},p_{s+2},...,p_{r}$ are distinct
prime numbers not equal to $p$. Let $I$ be an ideal of the ring $A\,=\mathbb{%
Z}_{p_{1}}\times \mathbb{Z}_{p_{2}}\times ...\times \mathbb{Z}_{p_{r}}$.
Then $I$ can have the following forms:

i) $I\simeq \underset{s-time}{\underbrace{\mathbb{Z}_{p}\times \mathbb{Z}%
_{p}\times ...\times \mathbb{Z}_{p}}}$;

ii) $I\simeq \underset{q-time}{\underbrace{\mathbb{Z}_{p}\times \mathbb{Z}%
_{p}\times ...\times \mathbb{Z}_{p}}},$ $q<s$;

iii) $I\simeq \underset{q-time}{\underbrace{\mathbb{Z}_{p}\times \mathbb{Z}%
_{p}\times ...\times \mathbb{Z}_{p}}}\times \mathbb{Z}_{t},q\leq s$ and $%
t=p_{i_{1}}p_{i_{2}}...p_{i_{l}}$, where $%
\{p_{i_{1}},p_{i_{2}},...,p_{i_{l}}\}\subseteq \{p_{s+1},p_{s+2},...,p_{r}\}$
are distinct, $1\leq l\leq r-s$.

If $I\simeq \underset{s-time}{\underbrace{\mathbb{Z}_{p}\times \mathbb{Z}%
_{p}\times ...\times \mathbb{Z}_{p}}}$, then the annihilator is $Ann\left(
I\right) \simeq \mathbb{Z}_{p_{s+1}}\times \mathbb{Z}_{p_{s+2}}\times
...\times \mathbb{Z}_{p_{r}}=\mathbb{Z}_{t}$, for $l=r-s$. The element $%
\alpha =\left( \underset{s-time}{\underbrace{\widehat{1},\widehat{1},...,%
\widehat{1}}},\underset{(r-s)-time}{\underbrace{\mathbf{0},\mathbf{0},...,%
\mathbf{0}}}\right) \in I$, where $\widehat{1}$ is the class modulo $p$ and $%
\mathbf{0}$ are classes in $\mathbb{Z}_{p_{s+1}},\mathbb{Z}_{p_{s+2}},...,%
\mathbb{Z}_{p_{r}}$. In the same time, the element $\beta =\left( \underset{%
s-time}{\underbrace{\widehat{0},\widehat{0},...,\widehat{0}}},\underset{%
(r-s)-time}{\underbrace{\mathbf{1},\mathbf{1},...,\mathbf{1}}}\right) \in
Ann\left( I\right) $, where $\widehat{0}$ is the class modulo $p$ and $%
\mathbf{1}$ are classes in $\mathbb{Z}_{p_{s+1}},\mathbb{Z}_{p_{s+2}},...,%
\mathbb{Z}_{p_{r}}$. With the above notations, we remark that $\alpha +\beta
=1,$ where $1=\left( \underset{s-time}{\underbrace{\widehat{1},\widehat{1}%
,...,\widehat{1}}},\underset{(r-s)-time}{\underbrace{\mathbf{1},\mathbf{1}%
,...,\mathbf{1}}}\right) $ is the unit element in $A$. Therefore $I$ and $%
Ann\left( I\right) $ are coprime.

If $I\simeq \underset{q-time}{\underbrace{\mathbb{Z}_{p}\times \mathbb{Z}%
_{p}\times ...\times \mathbb{Z}_{p}}},$ $q<s$, then $Ann\left( I\right)
\simeq \underset{(s-q)-time}{\underbrace{\mathbb{Z}_{p}\times \mathbb{Z}%
_{p}\times ...\times \mathbb{Z}_{p}}}\times \mathbb{Z}_{p_{s+1}}\times
\mathbb{Z}_{p_{s+2}}...\times \mathbb{Z}_{p_{r}}$. The element $\alpha
=\left( \underset{q-time}{\underbrace{\widehat{1},\widehat{1},...,\widehat{1}%
}},\underset{(s-q)-time}{\underbrace{\widehat{0},\widehat{0},...,\widehat{0}}%
,}\underset{(r-s)-time}{\underbrace{\mathbf{0},\mathbf{0},...,\mathbf{0}}}%
\right) \in I$, where $\widehat{0},\widehat{1}$ are the classes modulo $p$
and $\mathbf{0}$ are classes in $\mathbb{Z}_{p_{s+1}},\mathbb{Z}%
_{p_{s+2}},...,\mathbb{Z}_{p_{r}}$. In the same time, the element $\beta
=\left( \underset{q-time}{\underbrace{\widehat{0},\widehat{0},...,\widehat{0}%
}},\underset{(s-q)-time}{\underbrace{\widehat{1},\widehat{1},...,\widehat{1}}%
,}\underset{(r-s)-time}{\underbrace{\mathbf{0},\mathbf{0},...,\mathbf{0}}}%
\right) \in Ann\left( I\right) $, where $\widehat{0},\widehat{1}$ are the
classes modulo $p$ and $\mathbf{1}$ are classes in $\mathbb{Z}_{p_{s+1}},%
\mathbb{Z}_{p_{s+2}},...,\mathbb{Z}_{p_{r}}$. With the above notations, we
remark that $\alpha +\beta =1,$ where $1=\left( \underset{s-time}{%
\underbrace{\widehat{1},\widehat{1},...,\widehat{1}}},\underset{(r-s)-time}{%
\underbrace{\mathbf{1},\mathbf{1},...,\mathbf{1}}}\right) $ is the unit
element in $A$. Therefore, $I$ and $Ann\left( I\right) $ are coprime.

If $I\simeq \underset{q-time}{\underbrace{\mathbb{Z}_{p}\times \mathbb{Z}%
_{p}\times ...\times \mathbb{Z}_{p}}}\times \mathbb{Z}_{t},q\leq s$ and $%
t=p_{i_{1}}p_{i_{2}}...p_{i_{l}}$, where $%
\{p_{i_{1}},p_{i_{2}},...,p_{i_{l}}\}\subseteq \{p_{s+1},p_{s+2},...,p_{r}\}$
are distinct, $1\leq l\leq r-s$, then the annihilator is $Ann\left( I\right)
\simeq \underset{(s-q)-time}{\underbrace{\mathbb{Z}_{p}\times \mathbb{Z}%
_{p}\times ...\times \mathbb{Z}_{p}}}\times \mathbb{Z}_{u}$, where $u=\prod
p_{j}$, with $p_{j}$ distinct prime numbers $p_{j}\in
\{p_{s+1},p_{s+2},...,p_{r}\}-\{p_{i_{1}},p_{i_{2}},...,p_{i_{l}}\}$. The
element $\alpha =\left( \underset{q-time}{\underbrace{\widehat{1},\widehat{1}%
,...,\widehat{1}}},\underset{l-time}{\underbrace{\mathbf{1},\mathbf{1},...,%
\mathbf{1}},}\underset{(r-q-l)-time}{\underbrace{\mathbf{0},\mathbf{0},...,%
\mathbf{0}}}\right) \in I$, where $\widehat{1}$ is the class modulo $p$ and $%
\mathbf{0,1}$ are classes in $\mathbb{Z}_{p_{q+1}},\mathbb{Z}_{p_{q+2}},...,%
\mathbb{Z}_{p_{r}}$. In the same time, the element $\beta =\left( \underset{%
q-time}{\underbrace{\widehat{0},\widehat{0},...,\widehat{0}}},\underset{%
l-time}{\underbrace{\mathbf{0},\mathbf{0},...,\mathbf{0}},}\underset{%
(r-q-l)-time}{\underbrace{\mathbf{1},\mathbf{1},...,\mathbf{1}}}\right) \in
Ann\left( I\right) $, where $\widehat{0}$ is the class modulo $p$ and $%
\mathbf{0},\mathbf{1}$ are classes in $\mathbb{Z}_{p_{q+1}},\mathbb{Z}%
_{p_{q+2}},...,\mathbb{Z}_{p_{r}}$. With the above notations, we remark that
$\alpha +\beta =1,$ where $1=\left( \underset{q-time}{\underbrace{\widehat{1}%
,\widehat{1},...,\widehat{1}}},\underset{(r-q)-time}{\underbrace{\mathbf{1},%
\mathbf{1},...,\mathbf{1}}}\right) $ is the unit element in $A$.

Therefore $I$ and $Ann\left( I\right) $ are coprime. For the last part of
the proposition, it is clear that for the ring $A$ we have $\complement
_{r}^{0}+\complement _{r}^{1}+\complement _{r}^{2}+...+\complement _{r}^{r}=$
$2^{r}$ ideals.
\end{proof}

\begin{remark}
\label{Remark 3.27} Due to the characterization of finite abelian groups, it
results that the rings $A$ from the above proposition \textit{\ }are the
only finite commutative unitary ring with the property that for each $I\in
Id\left( A\right) $ the ideals $I$ and $Ann\left( I\right) $ are coprime.
\end{remark}

\begin{example}
\label{Example 3.28} 1) Let $A=\mathbb{Z}_{4}=\mathbb{Z}_{2^{2}}.$ The
ideals are $\{\widehat{0}\},\{\widehat{0},\widehat{2}\}$ and $\mathbb{Z}_{4}$%
. For $I=\{\widehat{0},\widehat{2}\}$, the annihilator is $Ann\left(
I\right) =$ $\{\widehat{0},\widehat{2}\}$, therefore $I$ and $Ann\left(
I\right) $ are not coprime ideals.

2) Let $A=\mathbb{Z}_{2}\times \mathbb{Z}_{2}=\{\left( \widehat{0},\widehat{0%
}\right) ,\left( \widehat{0},\widehat{1}\right) ,\left( \widehat{1},\widehat{%
0}\right) ,\left( \widehat{1},\widehat{1}\right) \}$. For $I=\{\left(
\widehat{0},\widehat{0}\right) ,\left( \widehat{0},\widehat{1}\right) \}$,
the annihilator is $Ann\left( I\right) =\{\left( \widehat{0},\widehat{0}%
\right) ,\left( \widehat{1},\widehat{0}\right) \}$. Obviously, $I$ and $%
Ann\left( I\right) $ are coprime ideals.

3) Let $A=\mathbb{Z}_{2}\times \mathbb{Z}_{2}\times \mathbb{Z}_{2}=$\newline
$=\{\left( \widehat{0},\widehat{0},\widehat{0}\right) ,\left( \widehat{0},%
\widehat{1},\widehat{0}\right) ,\left( \widehat{1},\widehat{0},\widehat{0}%
\right) ,\left( \widehat{1},\widehat{1},\widehat{0}\right) ,\left( \widehat{0%
},\widehat{0},\widehat{1}\right) ,\left( \widehat{1},\widehat{0},\widehat{1}%
\right) ,$\newline
$\left( \widehat{0},\widehat{1},\widehat{1}\right) ,\left( \widehat{1},%
\widehat{1},\widehat{1}\right) \}$. Let $I=\{\left( \widehat{0},\widehat{0},%
\widehat{0}\right) ,\left( \widehat{0},\widehat{1},\widehat{0}\right) \}$ be
an ideal. The annihilator is $Ann\left( I\right) =\{\left( \widehat{0},%
\widehat{0},\widehat{0}\right) ,\left( \widehat{1},\widehat{0},\widehat{0}%
\right) ,\left( \widehat{0},\widehat{0},\widehat{1}\right) ,\left( \widehat{1%
},\widehat{0},\widehat{1}\right) \}$. It results that $I$ and $Ann\left(
I\right) $ are coprime ideals and this is true for all ideals of $A$. We
remark that this ring has $8$ ideals.

4) Let $A=\mathbb{Z}_{2}\times \mathbb{Z}_{2}\times \mathbb{Z}_{3}\simeq
\mathbb{Z}_{2}\times \mathbb{Z}_{6}=\{\left( \widehat{0},\widehat{0},%
\overline{0}\right) ,\left( \widehat{0},\widehat{1},\overline{0}\right)
,\left( \widehat{1},\widehat{0},\overline{0}\right) ,\left( \widehat{0},%
\widehat{1},\overline{1}\right) ,$\newline
$\left( \widehat{0},\widehat{1},\overline{2}\right) ,\left( \widehat{1},%
\widehat{1},\overline{2}\right) ,\left( \widehat{0},\widehat{0},\overline{1}%
\right) ,\left( \widehat{1},\widehat{1},\overline{0}\right) ,\left( \widehat{%
1},\widehat{0},\overline{1}\right) ,$\newline
$\left( \widehat{0},\widehat{0},\overline{2}\right) ,\left( \widehat{1},%
\widehat{0},\overline{2}\right) ,\left( \widehat{1},\widehat{1},\overline{1}%
\right) \}.$ For $I=\{\left( \widehat{0},\widehat{0},\overline{0}\right)
,\left( \widehat{0},\widehat{1},\overline{0}\right) \}$, the annihilator is
\newline
$Ann\left( I\right) =\{\left( \widehat{0},\widehat{0},\overline{0}\right)
,\left( \widehat{1},\widehat{0},\overline{0}\right) ,\left( \widehat{0},%
\widehat{0},\overline{1}\right) ,\left( \widehat{1},\widehat{0},\overline{1}%
\right) ,\left( \widehat{0},\widehat{0},\overline{2}\right) ,\left( \widehat{%
1},\widehat{0},\overline{2}\right) \}$ and is coprime to $I.$ For $%
I=\{\left( \widehat{0},\widehat{0},\overline{0}\right) ,\left( \widehat{0},%
\widehat{1},\overline{0}\right) ,\left( \widehat{1},\widehat{0},\overline{0}%
\right) ,\left( \widehat{1},\widehat{1},\overline{0}\right) \}$, the
annihilator is \newline
$Ann\left( I\right) =\{\left( \widehat{0},\widehat{0},\overline{0}\right)
,\left( \widehat{0},\widehat{0},\overline{1}\right) ,\left( \widehat{0},%
\widehat{0},\overline{2}\right) \}$ and is coprime to $I.$

For $I=\{\left( \widehat{0},\widehat{0},\overline{0}\right) ,\left( \widehat{%
1},\widehat{0},\overline{0}\right) \}$, the annihilator is \newline
$Ann\left( I\right) =\{\left( \widehat{0},\widehat{0},\overline{0}\right)
,\left( \widehat{0},\widehat{1},\overline{0}\right) ,\left( \widehat{0},%
\widehat{0},\overline{1}\right) ,\left( \widehat{0},\widehat{1},\overline{1}%
\right) ,\left( \widehat{0},\widehat{0},\overline{2}\right) ,\left( \widehat{%
0},\widehat{1},\overline{2}\right) \}$ and is coprime to $I$, etc.

5) Let $A=\mathbb{Z}_{4}\times \mathbb{Z}_{3}\simeq \mathbb{Z}_{12}$. For $%
I=\{\widehat{0},\widehat{6}\}$, the annihilator is \newline
$Ann\left( I\right) =\{\widehat{0},\widehat{2},\widehat{4},\widehat{6},%
\widehat{8},\widehat{10}\}$ and $Ann\left( I\right) $ is not coprime to $I$.
\medskip

6) If $A$ is a finite integral domain, then $(Id(A),\cap ,+,Ann,0=\{0\},1=A)$
is a Boolean algebra. Indeed, every finite integral domain is a field, so $%
Id(A)=\{\{0\},A\}\simeq L_{2}$ is a Boolean algebra.

7) If $K$ is a field, then $Id(K)=\{\mathbf{0},K\}$ and $I$ and $Ann(I)$ are
coprime for every $I\in Id(K).$
\end{example}

\medskip

Using Corollary \ref{Corollary 3.13} and Proposition \ref{Proposition 3.26}
we obtain:

\begin{corollary}
\label{Corollary 3.30} \textit{If }$A$\textit{\ is a finite commutative ring
with }$\left\vert A\right\vert =n=p_{1}^{\alpha _{1}}\cdot ...\cdot
p_{r}^{\alpha _{r}},$\textit{\ then its set of ideals is an MV-algebra. Of
all its representations, only if }$A$\textit{\ is isomorphic to the ring }$%
\underset{\alpha _{1}-time}{\underbrace{\mathbb{Z}_{p_{1}}\times \mathbb{Z}%
_{p_{1}}\times ...\times \mathbb{Z}_{p_{1}}}\times ...\times }\underset{%
\alpha _{r}-time}{\underbrace{\mathbb{Z}_{p_{r}}\times \mathbb{Z}%
_{p_{r}}\times ...\times \mathbb{Z}_{p_{r}}}}$\textit{\ the lattice of its
ideals is a Boolean algebra.}
\end{corollary}

\medskip Using the equivalent characterizations for Boolean elements in
residuated lattices (see [BT; 03]), we obtain necessary and sufficient
conditions for rings in which the lattice of ideals is a Boolean algebra:

\begin{corollary}
\label{Corollary 3.32} \textit{Let }$A$\textit{\ be a commutative unitary
ring. The following assertions are equivalent:}

\textit{(i) }$(Id(A),\cap ,+,Ann,0=\{0\},1=A)$\textit{\ is a Boolean algebra;%
}

\textit{(ii) }$I$\textit{\ and }$Ann(I)$\textit{\ are coprime, for every }$%
I\in Id(A);$

\textit{(iii) }$(I:Ann(I))=I,$\textit{\ for every }$I\in Id(A);$

\textit{(iv) For }$J\in Id(A),$\textit{\ }$Ann(J)\subseteq J$\textit{\
implies }$J=A;$

\textit{(v) For }$I,J\in Id(A),$\textit{\ }$I\subseteq (J:Ann(J))$\textit{\
\ implies }$I\subseteq J;$

\textit{(vi) For }$I,J,K\in Id(A),$\textit{\ }$I\subseteq (J:Ann(K))$\textit{%
\ and }$J\subseteq K;$\textit{\ implies }$I\subseteq K;$

\textit{(vii ) For }$I,J\in Id(A),$\textit{\ }$(J:I)\subseteq I$\textit{\ \
implies }$I=A.$
\end{corollary}

Moreover, using some connections between Heyting algebras, MV-algebras and
involutive residuated lattices we obtain necessary and sufficient conditions
for which the lattice of ideals in commutative unitary rings is a Boolean
algebra. The following results can be obtained by the straightforward
calculations, using Proposition \ref{Proposition 3.33}:

\begin{corollary}
\label{Corollary 3.34}\textit{\ Let }$A$\textit{\ be a commutative unitary
ring. The following assertions are equivalent:}

\textit{(i) }$A$ is a Von Neumann regular ring in which
\begin{equation*}
Ann(Ann(I))=I\mathit{,}
\end{equation*}%
\textit{\ for every }$I\in Id(A);$

(ii) $A$ is a Von Neumann regular ring that satisfies the condition
\begin{equation*}
Ann(J)=\mathbf{0}\Rightarrow J=A,
\end{equation*}%
for $J\in Id(A);$

\textit{(iii) }$(Id(A),\cap ,+,Ann,0=\{0\},1=A)$\textit{\ is a Boolean
algebra.}
\end{corollary}

\section{\textbf{Connections with binary block codes}}

In [FV; 20] have been defined binary block codes associated to MV-algebras
and Wajsberg algebras and in [FHSV; 20] a classification of these algebras
was done.

One of the question which arise is what these codes really represent.

In Section 3, we proved that the lattice of ideals of a finite commutative
ring is a Boolean algebra or an MV-algebra, see Corollary \ref{Corollary
3.30}.

Using this result, in this section we construct binary block codes
associated to the lattice of ideals in a finite commutative ring $%
A=\{x_{1},x_{2},...,x_{n}\}$ with $n$ elements.

From Remark \ref{Remark 3.14}, $Id\left( A\right) ,$ the set of its ideals
of $A,$ \ has $\mathcal{N}_{A}$ elements (obviously, $\mathcal{N}_{A}\leq n$%
). From Corollary \ref{Corollary 3.30}, $Id\left( A\right) $ can be a
Boolean algebra or an MV-algebra (that is not Boolean) and it can generates
a binary block codes. To each ideal $I_{\alpha }$ from $Id\left( A\right) $
we associate a codeword $\alpha $ of length $n,\alpha =\left( \alpha
_{1},...,\alpha _{n}\right) ,\alpha _{i}=1$ if $x_{i}\in $ $I_{\alpha }$ and
$\alpha _{i}=0$, otherwise.

We denote with $\mathcal{C}_{A}$ the set of these codewords, which is a
binary block code. The obtained block code $\mathcal{C}_{A}=%
\{w_{1},w_{2},...,w_{\mathcal{N}_{A}}\}$, with lexicographic order is a
totally ordered set. On $\mathcal{C}_{A}$ we can define a Boolean algebra
structure, if $\mathcal{N}_{A}$ is a power of $2$, or an MV-algebra
structure, otherwise. This algebra generates a binary block code $\mathcal{C}%
_{2}$.

The code $\mathcal{C}_{2}$ is called the \textit{reduced code} of the code $%
\mathcal{C}_{A}$. It is clear that all finite commutative rings $A$ with the
set $Id\left( A\right) $ having the same number of elements, $\mathcal{N}%
_{A} $, generate the same reduced binary block code with $\mathcal{N}_{A}$
codewords of length $n$.

By using the above definitions from Section 2, we have the following
correspondence between the operations of an MV-algebra $M:$
\begin{equation*}
x\odot y=\left( x^{\ast }\oplus y^{\ast }\right) ^{\ast }\text{ and }x\oplus
y=\left( x^{\ast }\odot y^{\ast }\right) ^{\ast }.
\end{equation*}%
In the associated Wajsberg algebra, we have that
\begin{equation*}
x\rightarrow y=x^{\ast }\oplus y=\left( x\odot y^{\ast }\right) ^{\ast }%
\text{.}
\end{equation*}%
Since, from Corollary \ref{Corollary 3.30}, the associated lattice of ideals
in the ring $A$ is a Wajsberg algebra, then the Wajsberg implication is
\begin{equation*}
I\rightarrow J=Ann\left( I\otimes Ann\left( J\right) \right) ,
\end{equation*}%
for every $I,J\in Id(A).$

It is known that ([FHSV; 20]\thinspace ) if $S$ is a nonempty set and\textit{%
\ }$(W,\rightarrow ,^{\ast },1)$\textit{\ }is a Wajsberg algebra then:

\begin{definition}
\label{Definition 4.1.} A mapping $f:S\rightarrow W$ is called a \textit{%
W-function} on $S$ and a map $f_{w}:S\rightarrow \{0,1\},w\in W$, such that
\begin{equation*}
f_{w}\left( x\right) =1\text{, if and only if \ }w\rightarrow f\left(
x\right) =1\text{, for every }x\in S\text{,}
\end{equation*}%
is called a \textit{cut function} \textit{of the map} $f~.$ \newline
The subset
\begin{equation*}
S_{w}=\{x\in S~:~w\rightarrow f\left( x\right) =1\}\subseteq S
\end{equation*}
is called a \textit{cut subset} of the set $S$.
\end{definition}

If $f:S\rightarrow W$ \ is a W-function on $S$ and we define on $W$ the
binary relation
\begin{equation*}
\forall w_{1},w_{2}\in W,w_{1}\sim w_{2}~~\text{if~and~only~if~~}%
S_{w_{1}}=S_{w_{2}},
\end{equation*}%
then, this relation is an equivalence relation on $W$.

For every $w\in W,$ we denote by $\widetilde{w}$ the equivalence class of $w$%
.\medskip

Now, let $A$ be a finite commutative ring with $n$ ideals, $S=\{1,...,n\}$
be a nonempty set and\textit{\ }$(Id(A),\rightarrow ,Ann,1=A)$\textit{\ }be
the Wajsberg algebra of ideals (see Corollary \ref{Corollary 3.30}).

Using above notations, to each equivalence class $\widetilde{I}$ (with $I\in
Id(A))$, will correspond the codeword $f_{I}=I_{1}I_{2}...I_{n}$, with $%
I_{i}=j$, if and only if $f_{I}\left( i\right) =j,i\in S,j\in \{0,1\}$%
.\medskip

\begin{example}
\label{Example 4.2} We consider \ the commutative ring $\ A=(\mathbb{Z}%
_{n},+,\cdot ).$

\textbf{Case 1,} $n=4$\textbf{.}

i) For $A=\mathbb{Z}_{4}$, the lattice $Id\left( \mathbb{Z}_{4}\right) $ has
$3$ elements, $Id\left( \mathbb{Z}_{4}\right) =\{\widehat{0},\{\widehat{0},%
\widehat{2}\},\mathbb{Z}_{4}\}=\{O,R,E\}$. Since $AnnO=E$, $AnnE=O$, $AnnR=R$
\ and $R\otimes R=O,$ so using Corollaries \ref{Corollary 3.13} and \ref%
{Corollary 3.23}, the obtained lattice is a Wajsberg (MV) algebra, with the
implication given in the below table:
\begin{equation*}
\begin{tabular}{l|lll}
$\rightarrow $ & $O$ & $R$ & $E$ \\ \hline
$O$ & $E$ & $E$ & $E$ \\
$R$ & $R$ & $E$ & $E$ \\
$E$ & $O$ & $R$ & $E$%
\end{tabular}%
.
\end{equation*}

The code $\mathcal{C}_{A}$ attached to the algebra $Id\left( \mathbb{Z}%
_{4}\right) $ is $\mathcal{C}_{\mathbb{Z}_{4}}=\{0001,0101,1111\}$. Its
reduced code is $\mathcal{C}_{2}=\{111,011,001\}$.

ii) For $A=\mathbb{Z}_{2}\times \mathbb{Z}_{2}=\{\left( \widehat{0},\widehat{%
0}\right) ,\left( \widehat{0},\widehat{1}\right) ,\left( \widehat{1},%
\widehat{0}\right) ,\left( \widehat{1},\widehat{1}\right) \}$, we obtain the
lattice $Id\left( \mathbb{Z}_{2}\times \mathbb{Z}_{2}\right) =\{\left(
\widehat{0},\widehat{0}\right) ,\{\left( \widehat{0},\widehat{0}\right)
,\left( \widehat{0},\widehat{1}\right) \},\{\left( \widehat{0},\widehat{0}%
\right) ,\left( \widehat{1},\widehat{0}\right) \},\mathbb{Z}_{2}\times
\mathbb{Z}_{2}\}=\{O,R,B,E\}$, which is a Wajsberg (Boolean) algebra since $%
I^{2}=I,$ for every $I\in Id(\mathbb{Z}_{2}\times \mathbb{Z}_{2}),$ see
Corollary \ref{Corollary 3.23}. Also, $AnnO=E,AnnE=O,AnnR=B,AnnB=R$,\ thus,
the implication table is
\begin{equation*}
\begin{tabular}{l|llll}
$\rightarrow $ & $O$ & $R$ & $B$ & $E$ \\ \hline
$O$ & $E$ & $E$ & $E$ & $E$ \\
$R$ & $B$ & $E$ & $B$ & $E$ \\
$B$ & $R$ & $R$ & $E$ & $E$ \\
$E$ & $O$ & $R$ & $B$ & $E$%
\end{tabular}%
.
\end{equation*}%
The code associated to the lattice $Id\left( \mathbb{Z}_{2}\times \mathbb{Z}%
_{2}\right) $ is $\mathcal{C}_{\mathbb{Z}_{2}\times \mathbb{Z}%
_{2}}=\{1000,1100,1010,1111\}$. The reduced code is $\mathcal{C}%
_{2}=\{1111,0101,0011,0001\}$.

\textbf{Case 2}, $n=6$.

For $A=\mathbb{Z}_{6}$, the lattice $Id\left( \mathbb{Z}_{6}\right) $ has $4$
elements, $Id\left( \mathbb{Z}_{6}\right) =\{\widehat{0},\{\widehat{0},%
\widehat{3}\},\{\widehat{0},\widehat{2},\widehat{4}\},\mathbb{Z}%
_{6}\}=\{O,R,B,E\}$. Since $AnnO=E$, $AnnE=O$, $AnnR=B$, $AnnB=R$ and $I$ $%
^{2}=I,$ for every $I\in Id(\mathbb{Z}_{6}{}),$ the obtained lattice is a
Wajsberg (Boolean) algebra, with the implication table

\begin{equation*}
\begin{tabular}{l|llll}
$\rightarrow $ & $O$ & $R$ & $B$ & $E$ \\ \hline
$O$ & $E$ & $E$ & $E$ & $E$ \\
$R$ & $B$ & $E$ & $B$ & $E$ \\
$B$ & $R$ & $R$ & $E$ & $E$ \\
$E$ & $O$ & $R$ & $B$ & $E$%
\end{tabular}%
.
\end{equation*}%
\qquad

The code attached to the lattice $Id\left( \mathbb{Z}_{6}\right) $ is $%
\mathcal{C}_{\mathbb{Z}_{6}}=\{000001,001001,010101,$ $111111\}$ and the
reduced code is $\mathcal{C}_{2}=\{1111,0101,0011,0001\}$, the same as in
the Case 1, ii).

The case $A=\mathbb{Z}_{2}\times \mathbb{Z}_{3}=\{\left( \widehat{0},%
\overline{0}\right) $,$\left( \widehat{0},\overline{1}\right) $,$\left(
\widehat{0},\overline{2}\right) $,$\left( \widehat{1},\overline{0}\right) $,$%
\left( \widehat{1},\overline{1}\right) $,$\left( \widehat{1},\overline{2}%
\right) \}$ is similar with the above case, since $\mathbb{Z}_{2}\times
\mathbb{Z}_{3}\simeq $ $\mathbb{Z}_{6}$.

\textbf{Case 3}, $n=8$.

i) For $A=\mathbb{Z}_{8}$, the lattice $Id\left( \mathbb{Z}_{8}\right) $ has
$4$ elements, $Id\left( \mathbb{Z}_{8}\right) =\{\widehat{0},\{\widehat{0},%
\widehat{4}\},\{\widehat{0},\widehat{2},\widehat{4},\widehat{6}\},$ $\mathbb{%
Z}_{8}\}=\{O,R,B,E\}$. Since $AnnO=E$, $AnnE=O$, $AnnR=B$, $AnnB=R$ and $%
R\otimes R=O,$ $B\otimes B=R,$ using Corollary \ref{Corollary 3.23}, the
obtained lattice is a Wajsberg (MV) algebra, with the implication table
given in the below table:
\begin{equation*}
\begin{tabular}{l|llll}
$\rightarrow $ & $O$ & $R$ & $B$ & $E$ \\ \hline
$O$ & $E$ & $E$ & $E$ & $E$ \\
$R$ & $B$ & $E$ & $E$ & $E$ \\
$B$ & $R$ & $B$ & $E$ & $E$ \\
$E$ & $O$ & $R$ & $B$ & $E$%
\end{tabular}%
.
\end{equation*}%
The code attached to the lattice $Id\left( \mathbb{Z}_{8}\right) $ is $%
\mathcal{C}_{\mathbb{Z}_{8}}=\{00000001,00010001,01010101,$ $11111111\}$ and
the reduced code is $\mathcal{C}_{2}=\{1111,0111,0011,0001\}$.

ii) For $A=\mathbb{Z}_{2}\times \mathbb{Z}_{4}=\{\left( \widehat{0},%
\overline{0}\right) $, $\left( \widehat{0},\overline{1}\right) $, $\left(
\widehat{0},\overline{2}\right) $, $\left( \widehat{0},\overline{3}\right) $%
, $\left( \widehat{1},\overline{0}\right) $, $\left( \widehat{1},\overline{1}%
\right) $, $\left( \widehat{1},\overline{2}\right) $, $\left( \widehat{1},%
\overline{3}\right) \}$, the lattice of ideals is \newline
$Id\left( \mathbb{Z}_{2}\times \mathbb{Z}_{4}\right) $=$\{\left( \widehat{0},%
\overline{0}\right) $,$\{\left( \widehat{0},\overline{0}\right) $,$\left(
\widehat{0},\overline{1}\right) $,$\left( \widehat{0},\overline{2}\right) $,$%
\left( \widehat{0},\overline{3}\right) \}$,\newline
$\{\left( \widehat{0},\overline{0}\right) ,\left( \widehat{1},\overline{0}%
\right) ,\left( \widehat{0},\overline{2}\right) ,\left( \widehat{1},%
\overline{2}\right) \}$,$\{\left( \widehat{0},\overline{0}\right) ,\left(
\widehat{0},\overline{2}\right) \}$, $\{\left( \widehat{0},\overline{0}%
\right) ,\left( \widehat{1},\overline{0}\right) \}$, $\mathbb{Z}_{2}\times
\mathbb{Z}_{4}\}=\{O,B,D,R,C,E\}$. Since $AnnO=E$, $AnnE=O$, $AnnR=D$, $%
AnnB=C$, $AnnD=R$, $AnnC=B$ \ and $R\otimes R=O,$ using Corollary \ref%
{Corollary 3.23}, the obtained lattice is a Wajsberg (MV) algebra, with the
implication table given below:%
\begin{equation*}
\begin{tabular}{l|llllll}
$\rightarrow $ & $O$ & $R$ & $B$ & $C$ & $D$ & $E$ \\ \hline
$O$ & $E$ & $E$ & $E$ & $E$ & $E$ & $E$ \\
$R$ & $D$ & $E$ & $E$ & $D$ & $E$ & $E$ \\
$B$ & $C$ & $D$ & $E$ & $C$ & $D$ & $E$ \\
$C$ & $B$ & $B$ & $B$ & $E$ & $E$ & $E$ \\
$D$ & $R$ & $B$ & $B$ & $D$ & $E$ & $E$ \\
$E$ & $O$ & $R$ & $B$ & $C$ & $D$ & $E$%
\end{tabular}%
.
\end{equation*}%
The code attached to the lattice $Id\left( \mathbb{Z}_{2}\times \mathbb{Z}%
_{4}\right) $ is $\mathcal{C}_{\mathbb{Z}_{2}\times \mathbb{Z}%
_{4}}=\{00000001$, $00000101$, $00001111$, $00010001$, $01010101$,$%
11111111\} $ and the reduced code is $\mathcal{C}_{2}=\{111111$, $011011$, $%
001001$, $000111$, $000011$, $00000001\}$.

iii) For $A=\mathbb{Z}_{2}\times \mathbb{Z}_{2}\times \mathbb{Z}%
_{2}=\{\left( \widehat{0},\widehat{0},\widehat{0}\right) $, $\left( \widehat{%
0},\widehat{0},\widehat{1}\right) $, $\left( \widehat{0},\widehat{1},%
\widehat{0}\right) $, $\left( \widehat{0},\widehat{1},\widehat{1}\right) $, $%
\left( \widehat{1},\widehat{0},\widehat{0}\right) $, $\left( \widehat{1},%
\widehat{0},\widehat{1}\right) $, $\left( \widehat{1},\widehat{1},\widehat{0}%
\right) $, $\left( \widehat{1},\widehat{1},\widehat{1}\right) $ $\}$, we
obtain the lattice \ $Id\left( \mathbb{Z}_{2}\times \mathbb{Z}_{2}\times
\mathbb{Z}_{2}\right) =\{\left( \widehat{0},\widehat{0},\widehat{0}\right) $%
, $\{\left( \widehat{0},\widehat{0},\widehat{0}\right) ,\left( \widehat{0},%
\widehat{0},\widehat{1}\right) \}$, $\{\left( \widehat{0},\widehat{0},%
\widehat{0}\right) ,\left( \widehat{0},\widehat{1},\widehat{0}\right) \}$, $%
\{\left( \widehat{0},\widehat{0},\widehat{0}\right) ,\left( \widehat{1},%
\widehat{0},\widehat{0}\right) \}$, $\{\left( \widehat{0},\widehat{0},%
\widehat{0}\right) ,\left( \widehat{0},\widehat{0},\widehat{1}\right) $, $%
\left( \widehat{0},\widehat{1},\widehat{0}\right) ,\left( \widehat{0},%
\widehat{1},\widehat{1}\right) \}$, $\{\left( \widehat{0},\widehat{0},%
\widehat{0}\right) ,\left( \widehat{0},\widehat{0},\widehat{1}\right) $, $%
\left( \widehat{1},\widehat{0},\widehat{0}\right) ,\left( \widehat{1},%
\widehat{0},\widehat{1}\right) \}$, $\{\left( \widehat{0},\widehat{0},%
\widehat{0}\right) ,\left( \widehat{0},\widehat{1},\widehat{0}\right) $, $%
\left( \widehat{1},\widehat{0},\widehat{0}\right) ,\left( \widehat{1},%
\widehat{1},\widehat{0}\right) \}$, $\mathbb{Z}_{2}\times \mathbb{Z}%
_{2}\times \mathbb{Z}_{2}\}=\{O,X,Y,T,Z,U,V,E\}$, which is a Wajsberg
(Boolean) algebra. Since $AnnO=E$, $AnnE=O$, $AnnX=V$, $AnnY=U$, $AnnZ=T$, $%
AnnT=Z$, $AnnU=Y$, $AnnV=X$, the implication is given in the below table:%
\begin{equation*}
\begin{tabular}{l|llllllll}
$\rightarrow $ & $O$ & $X$ & $Y$ & $Z$ & $T$ & $U$ & $V$ & $E$ \\ \hline
$O$ & $E$ & $E$ & $E$ & $E$ & $E$ & $E$ & $E$ & $E$ \\
$X$ & $V$ & $E$ & $V$ & $E$ & $V$ & $E$ & $V$ & $E$ \\
$Y$ & $U$ & $U$ & $E$ & $E$ & $U$ & $U$ & $E$ & $E$ \\
$Z$ & $T$ & $U$ & $V$ & $E$ & $T$ & $U$ & $V$ & $E$ \\
$T$ & $Z$ & $Z$ & $Z$ & $Z$ & $E$ & $E$ & $E$ & $E$ \\
$U$ & $Y$ & $Z$ & $Y$ & $Z$ & $V$ & $E$ & $V$ & $E$ \\
$V$ & $X$ & $X$ & $Z$ & $Z$ & $U$ & $U$ & $E$ & $E$ \\
$E$ & $O$ & $X$ & $Y$ & $Z$ & $T$ & $U$ & $V$ & $E$%
\end{tabular}%
\text{.}
\end{equation*}%
The code attached to the lattice $Id\left( \mathbb{Z}_{2}\times \mathbb{Z}%
_{2}\times \mathbb{Z}_{2}\right) $ is $\mathcal{C}_{\mathbb{Z}_{2}\times
\mathbb{Z}_{2}\times \mathbb{Z}_{2}}=\{00000001$, $00000011$, $00000101$, $%
00010001$, $00001111$, $00110011$, $01010101$, $11111111\}$, which is
similar to its reduced code $\mathcal{C}_{2}$.
\end{example}

\begin{definition}
\label{Definition 4.3} ([Li; 99])Let $\mathcal{C}$ be a code. The \textit{%
Hamming distance} $d\left( c_{1},c_{2}\right) $, between two codewords $%
c_{1},c_{2}$ of the same length, is the number of positions in which the
corresponding symbols are different. The minimum Hamming distance of the
code $\mathcal{C}$, denoted $d_{H}$, is%
\begin{equation*}
d_{H}=\text{\textit{min}}\{d\left( c_{1},c_{2}\right) ,c_{1},c_{2}\in
\mathcal{C},c_{1}\not=c_{2}\}\text{.}
\end{equation*}

The code $\mathcal{C}$ is considered to be $k$-error detecting if and only
if $d_{H}$ between any two of its codewords is at least $k+1$. The code $%
\mathcal{C}$ is $k$-errors correcting if and only if $d_{H}$ between any two
of its codewords is at least $2k+1$. Therefore, if $d_{H}\geq 2$, the code $%
\mathcal{C}$ is an error detecting code. If $d_{H}\geq 3$, the code $%
\mathcal{C}$ is an error correcting code.
\end{definition}

\begin{remark}
\label{Remark 4.4} The codes generated by a Wajsberg (MV) algebra is not
error detecting, nor error correcting codes, since $d_{H}$ is always $1$.
\end{remark}

\begin{proposition}
\label{Proposition 4.5} Let $A$ \textit{be a finite commutative unitary
ring, }$A=\mathbb{Z}_{k_{1}}\times \mathbb{Z}_{k_{2}}\times ...\times
\mathbb{Z}_{k_{r}},$ \textit{where} $k_{i}=p_{i}^{\alpha _{i}}$, $p_{i}$
\textit{a prime number,} \textit{for all} $i\in \{1,2,...,r\}$.

\textit{i) If\ }$p_{i}\geq 3$\textit{, for all} $i\in \{1,2,...,r\}$\textit{%
, then the attached code} $\mathcal{C}_{A}$ \textit{is an error detecting
code.}

\textit{ii) If\ }$p_{i}\geq 5$\textit{, for all }$i\in \{1,2,...,r\}$\textit{%
, then the attached code} $\mathcal{C}_{A}$ \textit{is an error correcting
code.}
\end{proposition}

\begin{proof}
If $p_{i}\geq 3$, for all $i\in \{1,2,...,r\}$, we have that $d_{H}$ is
minimum $2$. If $p_{i}\geq 5$, for all $i\in \{1,2,...,r\}$, then $d_{H}$ is
minimum $3$.
\end{proof}

\begin{remark}
\label{Remark 4.6.} Let $W$ be a Wajsberg (MV or Boolean) algebra with$~%
\mathcal{N}_{A}=\overset{r}{\underset{i=1}{\prod }}\left( \alpha
_{i}+1\right) $ elements and $\mathcal{C}_{W}$ its attached binary block
code. There is a finite unitary commutative ring $A=\mathbb{Z}_{k_{1}}\times
\mathbb{Z}_{k_{2}}\times ...\times \mathbb{Z}_{k_{r}},$ where $%
k_{i}=p_{i}^{\alpha _{i}}$, $p_{i}$ a prime number, for all $i\in
\{1,2,...,r\}$, with lattice of ideals $Id\left( A\right) $ having $\mathcal{%
N}_{A}$ elements and its attached binary block code $\mathcal{C}_{A}$ such
that $\mathcal{C}_{W}$ is the reduced code of the code $\mathcal{C}_{A}$.
\end{remark}

\section{Generation of finite MV-algebras using finite commutative rings}

In this section, using the results obtained in Section 3, we present a way
to generate finite\ MV-algebras using finite commutative rings.

We recall that an MV-algebra is finite if and only if it is isomorphic to a
finite product of totally ordered MV-algebras, see [HR; 99].

Also, it is known that, on a finite totally ordered set there is only one
way to define an MV-algebra, see [FRT; 84]$.$

Now, let $n\geq 2$ be a natural number.

If we consider the decomposition of $n$ in factors greater than $1$, then
this decomposition is not unique. We denote by $\pi (n)$ the number of all
such decompositions.

We conclude that for every natural number $n\geq 2$ there are $\pi (n)$
non-isomorphic MV-algebras with $n$ elements which are obtained as a finite
product of totally ordered MV-algebras. Furthermore, there is only one
MV-algebra with $n$ elements which is a chain.

We deduce that:

\begin{proposition}
\label{Proposition 5.0} For every natural number $n\geq 2,$ there are $\pi
(n)+1$ non-isomorphic MV-algebras with $n$ elements.
\end{proposition}

\begin{example}
\label{Example 5.1} For $n=8,$ we have $8=2\cdot 4=2\cdot 2\cdot 2.$ Thus,
we have two types (up to an isomorphism) of MV-algebras with $8$ elements
obtained as a finite product of MV chains. In addition, we have only one
MV-algebra with $8$ elements which is a chain.Therefore, there are three
types of MV-algebras (up to an isomorphism) with $8$ elements.
\end{example}

\textbf{Table 1 }present a summary for the number of MV-algebras and Boolean
algebras with $n\leq 8$ elements:

\begin{equation*}
\text{\textbf{Table 1 }}
\end{equation*}

\medskip

\textbf{\ \ }%
\begin{tabular}{llllllll}
& $n=2$ & $n=3$ & $n=4$ & $n=5$ & $n=6$ & $n=7$ & $n=8$ \\
MV-alg & $1$ & $1$ & $2$ & $1$ & $2$ & $1$ & $3$ \\
Boole alg & $1$ & $-$ & $1$ & $-$ & $-$ & $-$ & $1$%
\end{tabular}%
. \medskip

An interesting thing is the relationship between the number of MV-algebras
with $n\geq 2$ elements and the number of MV-chains with $n$ elements (only
one) or the number of MV-algebras which are Boolean algebras (only one if $n$
is a power of $2$).

\medskip

Using Corollary \ref{Corollary 3.30}, if $A$ is a finite commutative ring,
its lattice of ideals is a Wajsberg (MV) algebra or a Wajsberg (Boolean)
algebra.

\textbf{Table 2 }present a basic summary for the structure of the lattice of
ideals in a finite and commutative ring $A$ with $2\leq n\leq 10$ elements:

\begin{equation*}
\text{\textbf{Table\ 2\ }}
\end{equation*}

\medskip

\ \
\begin{tabular}{lll}
$\left\vert A\right\vert \mathbf{=n}$ & $A$\textbf{\ is isomorphic to one of
the rings:} & $Id(A)$\textbf{\ is:} \\
$n=2$ & $Z_{2}$ & Boolean algebra \\
$n=3$ & $Z_{3}$ & Boolean algebra \\
$n=4$ & $Z_{4}$ & MV-algebra \\
& $Z_{2}\times Z_{2}$ & Boolean algebra \\
$n=5$ & $Z_{5}$ & Boolean algebra \\
$n=6$ & $Z_{6}\simeq Z_{3}\times Z_{2}$ & Boolean algebra \\
$n=7$ & $Z_{7}$ & Boolean algebra \\
$n=8$ & $Z_{8}$ & MV-algebra \\
& $Z_{4}\times Z_{2}$ & MV-algebra \\
& $Z_{2}\times Z_{2}\times Z_{2}$ & Boolean algebra \\
$n=9$ & $Z_{9}$ & MV-algebra \\
& $Z_{3}\times Z_{3}$ & Boolean algebra \\
$n=10$ & $Z_{10}\simeq Z_{5}\times Z_{2}$ & Boolean algebra%
\end{tabular}

\medskip

One interesting thing is that for any finite commutative ring, just for one
representation of this ring, the lattice of ideals is a Boolean algebra. In
all other cases, this lattice is an MV-algebra which is not a Boolean
algebra.

Thus, in order to generate all MV-algebras with $n\geq 2$ elements it is
suffices to find finite commutative rings with $n$ ideals.

\begin{remark}
\label{Remark 5.1} All finite MV-algebras (up to an isomorphism) with $n\geq
2$ elements correspond to finite commutative rings $A$ in which $\left\vert
Id(A)\right\vert =n.$
\end{remark}

\begin{lemma}
\label{Lemma 5.2} If $p\geq 2$ is a prime number and $k\geq 1$ is a natural
number, then $(Id(Z_{p^{k}}),\cap ,+,\otimes ,\rightarrow
,0=\{0\},1=Z_{p^{k}})$ is the only MV-chain (up to an isomorphism) with $k+1$
elements.
\end{lemma}

\begin{proof}
Obviously, the ring $(Z_{p^{k}},+,\cdot )$ has $k+1$ ideals: $I_{0}=\{0\},$ $%
I_{1}=\widehat{p^{k-1}}Z_{p^{k}},$ ..., $I_{k-2}=\widehat{p^{2}}Z_{p^{k}},$ $%
I_{k-1}=\widehat{p}Z_{p^{k}},$ $I_{k}=Z_{p^{k}}.$ Since $I_{0}\subseteq
I_{1}\subseteq I_{2}\subseteq ...\subseteq I_{k},$ using Corollary \ref%
{Corollary 3.30}, $Id(Z_{p^{k}})$ is an MV-chain. Moreover, for every $%
i,j\in \{0,...,k\}$ we have $I_{i}\rightarrow I_{j}=Z_{p^{k}}$ if $i\leq j$
and $I_{k-i+j}$ otherwise. Also, $I_{i}^{\ast }=Ann(I_{i})=$ $I_{k-i}$ for
every $i\in \{0,...,k\}.$ Moreover, for every $i,j\in \{0,...,k\},$
\begin{equation*}
I_{i}\oplus I_{j}=Ann(I_{i})\rightarrow I_{j}=I_{k-i}\rightarrow
I_{j}=Z_{p^{k}}\text{ if }k\leq i+j\text{ and }I_{i+j}\text{ otherwise.}
\end{equation*}
\end{proof}

\begin{example}
We generate all finite MV-algebras $M$ with $n=6=2\cdot 3$ elements (up to
an isomorphism). \ Thus, we have $\pi (6)+1=2$ types of non-isomorphic
MV-algebras: \ one MV-algebra is an MV-chain (for example, $Id(Z_{32})$
which has $6$ ideals, since $32=2^{5},$ see Lemma \ref{Lemma 5.2}) and one
MV-algebra is a product of \ MV-chains (for example, $Id(Z_{2}\times Z_{4})$
which has $2\cdot 3=6$ ideals).

\textbf{Case 1.} In $(\mathbb{Z}_{32},+,\cdot )$ the lattice of ideals has $%
6 $ elements: \newline
$Id\left( \mathbb{Z}_{32}\right) $=$\{I_{0}=\{\widehat{0}\}$, $\ I_{1}=%
\widehat{16}\mathbb{Z}_{32}=\{\widehat{0},\widehat{16}\},$ $I_{2}=\widehat{8}%
\mathbb{Z}_{32}=\{\widehat{0},\widehat{8},\widehat{16},\widehat{24}\},$ $%
I_{3}=\widehat{4}\mathbb{Z}_{32}=\{\widehat{0},\widehat{4},\widehat{8},...,%
\widehat{28}\},$ $I_{4}=\widehat{2}\mathbb{Z}_{32}=\{\widehat{0},\widehat{2},%
\widehat{4},...,\widehat{30}\}$, $,$ $I_{5}=\mathbb{Z}_{32}\}$. Using
Corollary \ref{Corollary 3.23}, since $I_{1}^{2}=I_{0},$ the obtained
lattice is an MV chain $(M,\oplus ,^{\ast },O)$ \ with $M=\{O,B,D,R,C,E\}$
in which $O\leq B\leq $ $D\leq R\leq C\leq E$ and the addition table is
given below:%
\begin{equation*}
\begin{tabular}{l|llllll}
$\oplus $ & $O$ & $R$ & $B$ & $C$ & $D$ & $E$ \\ \hline
$O$ & $O$ & $R$ & $B$ & $C$ & $D$ & $E$ \\
$R$ & $R$ & $E$ & $C$ & $E$ & $E$ & $E$ \\
$B$ & $B$ & $C$ & $D$ & $E$ & $R$ & $E$ \\
$C$ & $C$ & $E$ & $E$ & $E$ & $E$ & $E$ \\
$D$ & $D$ & $E$ & $R$ & $E$ & $C$ & $E$ \\
$E$ & $E$ & $E$ & $E$ & $E$ & $E$ & $E$%
\end{tabular}%
.
\end{equation*}%
Using Lemma \ref{Lemma 5.2}, in this MV-algebra, $O^{\ast }=E$, $E^{\ast }=O$%
, $R^{\ast }=D$, $B^{\ast }=C$, $D^{\ast }=R$ and $C^{\ast }=B$. \

\textbf{Case 2.} In $\mathbb{Z}_{2}\times \mathbb{Z}_{4}$ the lattice of
ideals has $6$ elements, see Example \ref{Example 4.2}, Case 3 (ii). \newline
Using Corollary \ref{Corollary 3.30}, we obtain an MV algebra $(M,\oplus
,^{\ast },O)$ \ with $M=\{O,B,D,R,C,E\}$ in which $O\leq R,C\leq $ $D\leq
E,O\leq R\leq B\leq E$ and $C,R$ respective $D,B$ are incomparable.

The addition table is given below:%
\begin{equation*}
\begin{tabular}{l|llllll}
$\oplus $ & $O$ & $R$ & $B$ & $C$ & $D$ & $E$ \\ \hline
$O$ & $O$ & $R$ & $B$ & $C$ & $D$ & $E$ \\
$R$ & $R$ & $B$ & $B$ & $D$ & $E$ & $E$ \\
$B$ & $B$ & $B$ & $B$ & $E$ & $E$ & $E$ \\
$C$ & $C$ & $D$ & $E$ & $C$ & $D$ & $E$ \\
$D$ & $D$ & $E$ & $E$ & $D$ & $E$ & $E$ \\
$E$ & $E$ & $E$ & $E$ & $E$ & $E$ & $E$%
\end{tabular}%
.
\end{equation*}%
Since $^{\ast }=Ann$ $\ $\ in this $MV-$algebra, $O^{\ast }=E$, $E^{\ast }=O
$, $R^{\ast }=D$, $B^{\ast }=C$, $D^{\ast }=R$ and $C^{\ast }=B.$
\end{example}

In the following, in \textbf{Table 3,} using Corollaries \ref{Corollary 3.23}%
, \ref{Corollary 3.30} and Lemma \ref{Lemma 5.2}, we sum briefly describe a
way to generate finite\ MV-algebras with $2\leq n\leq 8$ elements.

In Table 3, the first column contains the number $n$ of elements of the
finite MV-algebra $M.$ The second column corresponds to the number of
MV-algebras with $n$ elements. The third column contains the rings $A$ which
generates these MV-algebras and the subvariety of MV algebras to which the
lattice of ideals $Id(A)$ belongs. \ In this column, $p\geq 2$ is a prime
number. \medskip

\begin{equation*}
\text{\textbf{Table 3:}}
\end{equation*}

\textbf{\ }%
\begin{tabular}{lll}
$\left\vert M\right\vert \mathbf{=n}$ & \textbf{Nr of MV} & \textbf{Rings
which generates MV} \\
$n=2$ & $1$ & $\mathbb{Z}_{p}$ (Boole chain) \\
$n=3$ & $1$ & $\mathbb{Z}_{p^{2}}$ (MV chain) \\
$n=4$ & $2$ & $\mathbb{Z}_{p^{3}}$ (MV chain) and $\mathbb{Z}_{p}\times
\mathbb{Z}_{p}$ (Boole) \\
$n=5$ & $1$ & $\mathbb{Z}_{p^{4}}$ (MV chain) \\
$n=6$ & $2$ & $\mathbb{Z}_{p^{5}}$ (MV chain) and $\mathbb{Z}_{p}\times
\mathbb{Z}_{p^{2}}$ (MV) \\
$n=7$ & $1$ & $\mathbb{Z}_{p^{6}}$ (MV chain) \\
$n=8$ & $3$ & $\mathbb{Z}_{p^{7}}$ (MV chain) and $\mathbb{Z}_{p}\times
\mathbb{Z}_{p^{3}}$ (MV) and $\mathbb{Z}_{p}\times \mathbb{Z}_{p}\times
\mathbb{Z}_{p}$ (Boole)%
\end{tabular}

\begin{equation*}
\end{equation*}

\textbf{Conclusions.} MV-algebras are algebraic structures corresponding to

\L ukasiewicz $\infty -$ valued propositional logic. In this paper using the
connections between some subvarieties of residuated lattices, we present a
way to generate all (up to an isomorphism) finite MV-algebras using rings.
For this, we investigated some properties of the lattice of ideals in
commutative rings and we show that, for finite rings of the form 
$A=\mathbb{Z}_{k_{1}}\times \mathbb{Z}_{k_{2}}\times ...\times \mathbb{Z}%
_{k_{r}},$ where $k_{i}=p_{i}^{\alpha _{i}}$ and $p_{i}$ a prime number, for
all $i\in \{1,2,...,r\}$, this lattice is a Boolean algebra or an MV-algebra
(which is not Boolean). As a further research, we intend to continue this
study by extended it to other classes of logical algebras. Also we will
investigate the codes attached to such structures, since we can have here
two directions of research. First direction, by studying of some logical
algebras, we will try to find when their attached codes are good and
performing. Logical algebras can be a way to define codes. We are looking
for those structures which give us codes with good parameters. For the first
time we find such a structure, attached to a logical algebra, on which we
can find codes with Hamming distance greater that $3$ (see Proposition \ref%
{Proposition 4.5} ). The second direction is the reverse of the first
direction. The study of the codes can give us new properties and
applications of logical algebras.

\begin{equation*}
\end{equation*}

\textbf{Acknowledgments.} The authors thank the referees for their
suggestions and remarks which helped us to improve this paper.

\begin{equation*}
\end{equation*}%
\textbf{References}%
\begin{equation*}
\end{equation*}

[AAT; 96] Abujabal, H.A.S., Aslam, M., Thaheem, A.B., \textit{A
representation of bounded commutative BCK-algebras}, Internat. J. Math. \&
Math. Sci., 19(4)(1996), 733-736.

[BD; 74] Balbes, R., Dwinger, P., \textit{Distributive lattices}, Columbia,
Missouri: University of Missouri Press. XIII, 1974.

[BN; 09] Belluce, L.P., Di Nola, A., \textit{Commutative rings whose ideals
form an MV-algebra}, Math. Log. Quart., \textbf{55} (5) (2009), 468-486.

[BNM; 10] Belluce, L.P., Di Nola, A., Marchioni, E., (2010) \textit{Rings
and G\"{o}del algebras}. Algebra Univ 64(1--2) (2010), 103--116.

[Bl; 53] Blair, R.L, \textit{Ideal lattices and the structure of rings},
Trans. Am. Math. Soc., 75(1953), 136--153.

[BT; 03] \ Blount, K., Tsinakis, C., \textit{The structure of residuated
lattices}, Internat. J. Algebra Comput., \textbf{13} (4) (2003), 437-461.

[BP; 02] Busneag, D., Piciu, D., \textit{Lectii de algebra}, Ed.
Universitaria, Craiova, 2002.

[COM; 00] Cignoli, R.; D'Ottaviano, I.M.L.; Mundici, D. \textit{Algebraic
Foundations of many-valued Reasoning}. Trends in Logic-Studia Logica Library
7, Dordrecht: Kluwer Acad. Publ.\textbf{\ 2000}.\medskip

[CL; 19] Chajda, I., L\"{a}nger, H., \textit{Commutative rings whose ideal
lattices are complemented,} Asian-European J Math 3:1950039 (2019).

[CHA; 58] Chang, C.C.,\textit{\ Algebraic analysis of many-valued logic},
Trans. Amer. Math. Soc. 88(1958), 467-490.

[Di; 38] Dilworth, R.P., \textit{Abstract residuation over lattices}, Bull.
Am. Math. Soc. 44(1938), 262--268.

[FHSV; 20] Flaut, C., Hoskova-Mayerova, S., Saeid, A.B., Vasile, R., \textit{%
Wajsberg algebras of} \textit{order} $n(n\leq 9)$, Neural Computing and
Applications, 32(2020), 13301--13312.

[FV; 20] Flaut, C., Vasile, R., \textit{Wajsberg algebras arising from
binary block codes}, Soft Computing, 24(2020), 6047-6058.

[FRT; 84] Font, J., M., Rodriguez, A., J., Torrens, A., \textit{Wajsberg
Algebras}, Stochastica, 8(1) (1984), 5-30.

[GDCK; 10] Van Gasse, B., Deschrijver, G., Cornelis, C., Kerre, E., \textit{%
Filters of residuated lattices and triangle algebras}, Inf. Sci., \textbf{180%
} (16) (2010), 3006-3020.

[HE; 75] Hernstein, I.N.,\textit{\ Topics in algebra}, 2end edition, John
Wiley\&Son, New York, 1975.

[HR; 99] H\H{o}hle, U., Rodabaugh, S. E., Mathematics of fuzzy sets:logic,
topology and measure theory, Springer, Berlin, 1999.

[I; 09] Iorgulescu, A., \textit{Algebras of logic as BCK algebras, }A.S.E.,
Bucharest, 2009.

[Li; 99] Lint J.H., Introduction to Coding Theory, third edition, Graduate
Texts in Mathematics, 86, Springer Verlag, Berlin, 1999.

[Me-Ju; 94] Meng, J., Jun, Y. B., \textit{BCK-algebras}, Kyung Moon Sa Co.
Seoul, Korea, 1994.

[Mu; 07] Mundici, D., \textit{MV-algebras-a short tutorial}, Department of
Mathematics \textit{Ulisse Dini}, University of Florence, 2007.

[Pi; 07] Piciu, D., \textit{Algebras of Fuzzy Logic}, Editura Universitaria,
Craiova, 2007.

[TT; 22] Tchoffo Foka, S. V., Tonga, M., \textit{Rings and residuated
lattices whose fuzzy ideals form a Boolean algebra}, Soft Comput., 26 (2022)
535-539.

[T; 99] Turunen, E., \textit{Mathematics Behind Fuzzy Logic}.
Physica-Verlag, \textbf{1999}.

[WD; 39] Ward, M., Dilworth, R.P., \textit{Residuated lattices}, Trans. Am.
Math. Soc. 45(1939), 335--354.

\begin{equation*}
\end{equation*}

Cristina Flaut

{\small Faculty of Mathematics and Computer Science, Ovidius University,}

{\small Bd. Mamaia 124, 900527, Constan\c{t}a, Rom\^{a}nia,}

{\small \ http://www.univ-ovidius.ro/math/}

{\small e-mail: cflaut@univ-ovidius.ro; cristina\_flaut@yahoo.com}

\bigskip

Dana Piciu

{\small Faculty of \ Science, University of Craiova, }

{\small A.I. Cuza Street, 13, 200585, Craiova, Romania,}

{\small http://www.math.ucv.ro/dep\_mate/}

{\small e-mail: dana.piciu@edu.ucv.ro, piciudanamarina@yahoo.com}

\end{document}